\documentclass[12pt,a4paper]{amsart}
\usepackage{esint}
\usepackage{amssymb}
\usepackage{MnSymbol}
\usepackage{esint}
\usepackage{comment}
\usepackage{graphicx}

\usepackage[
hypertexnames=false, colorlinks, 
linkcolor=black,citecolor=black,urlcolor=black, linktocpage=true
]%
{hyperref} 
\hypersetup{bookmarksdepth=3}

	\normalbaselines						  %

\newtheorem{thm}{Theorem}[section]
\newtheorem{lm}[thm]{Lemma}
\newtheorem{cor}[thm]{Corollary}
\newtheorem{prop}[thm]{Proposition}

\theoremstyle{definition}

\newtheorem*{df*}{Definition}

\theoremstyle{remark}
\newtheorem{rem}[thm]{Remark}
\newtheorem*{rem*}{Remark}

\numberwithin{equation}{section}

\newcommand{\ci}[1]{_{ {}_{\scriptstyle #1}}}
\newcommand{\ti}[1]{_{\scriptstyle \text{\rm #1}}}

\newcommand{\dd}{{\mathrm{d}}}

\newcommand{\mathd}{\mathrm{d}}

\newcommand{\cD}{\mathcal{D}}

\newcommand{\cX}{\mathcal{X}}

\newcommand{\cH}{\mathcal{H}}

\newcommand{\cW}{\mathcal{W}}
\newcommand{\cV}{\mathcal{V}}

\newcommand{\C}{\mathbb{C}}
\newcommand{\R}{\mathbb{R}}
\newcommand{\Z}{\mathbb{Z}}

\newcommand{\Q}{\mathbb{Q}}

\newcommand{\cU}{\mathcal{U}}

\newcommand{\ch}{\operatorname{ch}}

\newcommand{\1}{\mathbf{1}}

\newcommand{\wt}{\widetilde}

\newcommand{\cz}{Calder\'{o}n--Zygmund\ }

\newcommand{\La}{\langle }
\newcommand{\Ra}{\rangle }

\newcommand{\bs}[1]{\boldsymbol{#1}}
\newcommand{\bfD}{\bs{\mathcal{D}}}

\newcommand{\fdot}{\,\cdot\,}

\newcommand{\supp}{\operatorname{supp}}

\DeclareMathOperator*{\esssup}{ess\,sup}

\newcommand{\dist}{\mathrm{dist}}
\newcommand{\fourier}[1]{\mathcal{F}\left[{#1}\right]}

\def\cyr{\fontencoding{OT2}\fontfamily{wncyr}\selectfont}
\DeclareTextFontCommand{\textcyr}{\cyr}

\newenvironment{entry}
{\begin{list}{X}%
		{%
			\setlength{\labelwidth}{55pt}%
			\setlength{\leftmargin}{\labelwidth}
			\addtolength{\leftmargin}{\labelsep}%
			\setlength{\itemsep}{.4pc}
	}%
}%
{\end{list}}

\title{Matrix-weighted little BMO spaces in two parameters}
\author{S.~Kakaroumpas and O.~Soler i Gibert}

\begin{document}

\begin{abstract}
In this paper we set up a theory of two-matrix weighted little BMO in two parameters. We prove that being a member of this class is equivalent to belonging uniformly in each variable to two-matrix weighted (one-parameter) BMO, a class studied extensively by J.~Isralowitz, S.~Pott, S.~Treil and others. Using this equivalence, we deduce lower and upper bounds in terms of the two-matrix weighted little BMO norm of the symbol for the norm of commutators with Journ\'e operators. 
\end{abstract}

\maketitle

\setcounter{tocdepth}{1}
\tableofcontents
\setcounter{tocdepth}{2}

\section*{Notation}

\begin{entry}
\item[$\1\ci{E}$] characteristic function of a set $E$;

\item[$\mathd x$] integration with respect to Lebesgue measure; 

\item[$|E|$] $n$-dimensional Lebesgue measure of a measurable set $E\subseteq\R^n$;

\item[$\La f\Ra\ci{E}$] average with respect to Lebesgue measure, $\La f\Ra\ci{E}:=\frac{1}{|E|}\int_{E}f(x)\mathd x$;

\item[$\strokedint_{E} \mathd x$] integral with respect to the normalized Lebesgue measure on a set $E$ of positive finite measure, $\strokedint_{E} f(x)\mathd x := \frac{1}{|E|} \int_{E} f(x)\mathd x = \La f\Ra\ci{E}$;

\item[$|e|$] usual Euclidean norm of a vector $e\in\C^d$;

\item[$\La e,f\Ra$] usual Euclidean pairing of vectors $e,f\in\C^d$;

\item[$M_{d}(\C)$] set of all $d\times d$-matrices with complex entries;

\item[$|A|$] usual matrix norm (i.e.~largest singular value) of a matrix $A\in M_{d}(\C)$;

\item[$I_{d}$] the identity $d\times d$-matrix;

\item[$A^{\ast}$] adjoint (complex-conjugate transpose) matrix to a matrix $A\in M_{d}(\C)$;

\item[$\Vert f \Vert\ci{\cX}$] norm of the element $f$ in a Banach space $\cX;$

\item[$\mu(E)$] $\mu(x)\mathd x$-measure of a set $E$, $\mu(E):=\int_{E}\mu(x)\mathd x$;

\item[$L^{p}(W)$] matrix-weighted Lebesgue space, $\|f\|\ci{L^p(W)}^p := \int|W(x)^{1/p}f(x)|^p \mathd x$;

\item[$(f,g)$] usual $L^2$-pairing, $(f,g) := \int \La f(x),g(x)\Ra \mathd x$, where $f,g$ take values in $\C^d$;

\item[$p'$] H\"{o}lder conjugate exponent to $p$, $1/p+1/p'=1$;

\item[$W'$] Muckenhoupt dual weight to weight $W$ corresponding to exponent $p$, $W':=W^{-p'/p}=W^{-1/(p-1)}$ (the exponent $p$ will always be clear from the context); 

\item[$I_{-},\,I_{+}$] left, respectively right half of an interval $I\subseteq\R$.
\end{entry}

The notation $x\lesssim\ci{a,b,\ldots}  y$ means $x\leq Cy$ with a constant $0<C<\infty$ depending \emph{only} on the quantities $a, b, \ldots$; the notation $x\gtrsim\ci{a,b,\ldots} y$ means $y\lesssim\ci{a,b,\ldots} x.$
We use  $x\sim\ci{a,b,\ldots} y$ if \emph{both} $x\lesssim\ci{a,b,\ldots} y$ and $x\gtrsim\ci{a,b,\ldots} y$ hold.

\section{Introduction}
The study of characterizations of the space $\text{BMO}(\R^n)$ is
a classical topic of interest due to the role of this space as a natural substitute
of the space $L^\infty$ in many areas of analysis.
Recall that the space $\text{BMO}(\R^n)$ consists of all locally integrable functions $b$ on $\R^n$ such that
\begin{equation*}
\Vert b\Vert\ci{\text{BMO}(\R^n)}:=\sup_{Q}\strokedint_{Q}|b(x)-\La b\Ra\ci{Q}|\dd x<\infty,
\end{equation*}
where the supremum is taken over all cubes $Q$ with sides parallel to the coordinate axes in $\R^n$, and where $\langle b \rangle_Q$ denotes the average of $b$ over $Q$. 
One such result is a theorem of Z.~Nehari~\cite{nehari},
which states that the norm of the Hankel operator $\mathrm{H}_b$ with symbol $b$
acting on the Hardy space $\cH^2$ is equivalent to the $\text{BMO}$ norm of $b.$
The real variable version of this theorem is the commutator theorem~\cite{coifman-rochberg-weiss} of
R.~R.~Coifman, R.~Rochberg and G.~Weiss, which relates
the norm of commutators between Riesz transforms and multiplication by a function $b$
with the $\text{BMO}$ norm of $b.$

Since then, the knowledge on $L^p$ boundedness of commutators between
general Calderón-Zygmund operators and multiplication by a function
and its relation to different versions of the space $\text{BMO}$ has been greatly extended.
In this article, we will consider various generalizations of the space $\text{BMO}$
at the same time.

The main goal of this paper is to develop a theory of two matrix-weighted little BMO. Let $1<p<\infty$. Consider two biparameter $d\times d$-matrix valued $A_p$ weights $U,V$ on $\R^n\times\R^m$. Let $B\in L^1\ti{loc}(\R^{n+m};M_{d}(\C)),$ that is a locally integrable function on $\R^{n+m}$ that takes values on the complex $d \times d$ matrices, and where by locally integrable we mean that the matrix norm $|B(x)|$ is a locally integrable real valued function. We define the two-matrix weighted little BMO norm $\Vert B\Vert\ci{\text{bmo}(U,V,p)}$ by
\begin{equation*}
\Vert B\Vert\ci{\text{bmo}(U,V,p)}:=\sup_{R}\left(\strokedint_{R}|V(x)^{1/p}(B(x)-\La B\Ra\ci{R})\cU\ci{R}^{-1}|^{p}\mathd x\right)^{1/p},
\end{equation*}
where the supremum is taken over all rectangles $R$ in $\R^n\times\R^m$ (with sides parallel to the coordinate axes), and all reducing operators are taken with respect to exponent $p$ (see Section~\ref{subsec:ReducingOperators} for definitions). We also define
\begin{equation*}
\Vert B\Vert\ci{\wt{\text{bmo}}(U,V,p)}:=\sup_{R}\strokedint_{R}|\cV\ci{R}(B(x)-\La B\Ra\ci{R})\cU\ci{R}^{-1}|\mathd x,
\end{equation*}
where the supremum is taken over all rectangles $R$ in $\R^n\times\R^m$ (with sides parallel to the coordinate axes), and all reducing operators are taken with respect to exponent $p$.
Our first result states that both resulting spaces are equivalent.

\begin{thm}
\label{thm:equivalences little bmo}
Let $p,U,V,B$ be as above. Then, there holds
\begin{equation*}
\Vert B\Vert\ci{\emph{bmo}(U,V,p)}\sim\Vert B\Vert\ci{\wt{\emph{bmo}}(U,V,p)},
\end{equation*}
where all implied constants depend only on $n,m,d,p,[U]\ci{A_p(\R^m\times\R^n)}$ and $[V]\ci{A_p(\R^m\times\R^n)}$.
\end{thm}

Using Theorem \ref{thm:equivalences little bmo} we deduce lower and upper bounds for commutators with Journ\'e operators.

\begin{thm}
  \label{thm:MatrixBMOLowerBound}
  Consider a function $B \in L^1\ti{loc}(\R^{n+m};M_{d}(\C)),$ let $1 < p < \infty$ and let $U$ and $V$ be biparameter $d \times d$-matrix valued $A_p$ weights. Let $R_j^{i},$ be the Riesz transform acting on the $i$-th variable and in the $j$-direction with respect to that variable. Then, there holds
  \begin{equation*}
  \Vert B\Vert\ci{\emph{bmo}(U,V,p)} \lesssim
  \max_{\substack{1\leq j\leq n\\ 1\leq k \leq m}} \Vert[R^1_jR^2_k,B]\Vert\ci{L^{p}(U)\rightarrow L^{p}(V)},
  \end{equation*}
  where the implied constant depends only on $n,m,d,p,[U]\ci{A_p(\R^n\times\R^m)}$ and $[V]\ci{A_p(\R^n\times\R^m)}$. 
\end{thm}

\begin{thm}
\label{thm:MatrixBMOUpperBound}
Let $T$ be any Journ\'e operator on $\R^n\times\R^m$. Let $1<p<\infty$, let $U,V$ be biparameter $d\times d$-matrix valued $A_p$ weights on $\R^n\times\R^m$, and let $B\in L^1\ti{loc}(\R^{n+m};M_d(\C))$. Then
\begin{equation*}
\Vert[T,B]\Vert\ci{L^{p}(U)\rightarrow L^{p}(V)}\lesssim\Vert B\Vert\ci{\emph{bmo}(U,V,p)},
\end{equation*}
where the implied constant depends only on $T,n,m,p,d,[U]\ci{A_p(\R^n\times\R^m)}$ and $[V]\ci{A_p(\R^n\times\R^m)}$.
\end{thm}

\subsection{Historical background}
Since we take various generalizations of the classical BMO space at once, we review the definitions and known related results for each of these separately.

\textbf{Weighted BMO spaces.} B.~Muckenhoupt and R.~L.~Wheeden \cite{muckenhoupt-wheeden} considered and studied the weighted BMO norm
\begin{equation}
\label{weighted BMO}
\Vert b\Vert\ci{\text{BMO}(\nu)}:=\sup_{Q}\frac{1}{\nu(Q)}\int_{Q}|b(x)-\La b\Ra\ci{Q}|\dd x,
\end{equation}
where the supremum is taken over all cubes $Q$ in $\R^n$ and $\nu$ is a weight. Such weighted BMO spaces characterize the boundedness of commutators over weighted Lebesgue spaces. Concretely, I.~Holmes, M.~Lacey and B.~D.~Wick \cite{holmes-lacey-wick} showed, extending earlier results~\cite{bloom},~\cite{coifman-rochberg-weiss} for the unweighted case, that if $1<p<\infty$, $\mu,\lambda$ are $A_p$ weights on $\R^n$ and $\nu=\mu^{1/p}\lambda^{-1/p}$, then for any \cz operator $T$ on $\R^n$ one has the inequality
\begin{equation*}
\Vert [T,b]\Vert\ci{L^{p}(\mu)\rightarrow L^{p}(\lambda)}\leq C\Vert b\Vert\ci{\text{BMO}(\nu)},
\end{equation*}
where the constant $0<C<\infty$ depends on $n,p,T$ and $[\mu]\ci{A_p(\R^n)},[\lambda]\ci{A_p(\R^n)}$, and if $R^{(1)}, \ldots, R^{(n)}$ are the Riesz transforms on $\R^n$, then
\begin{equation*}
\sum_{i=1}^{n}\Vert[R^{(i)},b]\Vert\ci{L^{p}(\mu)\rightarrow L^{p}(\lambda)}\geq c\Vert b\Vert\ci{\text{BMO}(\nu)},
\end{equation*}
where the constant $0<c<\infty$ depends only on $n,p$ and $[\mu]\ci{A_p(\R^n)},[\lambda]\ci{A_p(\R^n)}$. The commutator $[T,b]$ acts here by
\begin{equation*}
[T,b](f):=T(bf)-b T(f).
\end{equation*}
We refer to Subsection \ref{s:Ap} for precise definition of $A_p$ weights. One of the crucial results of \cite{holmes-lacey-wick} was the following weighted John--Nirenberg type inequality,
\begin{equation}
\label{holmes-lacey-wick John-Nirenberg}
c\Vert b\Vert\ci{\text{BMO}(\nu)}\leq\sup_{Q}\left(\frac{1}{\mu(Q)}\int_{Q}|b(x)-\La b\Ra\ci{Q}|^p\lambda (x)\dd x\right)^{1/p}\leq C\Vert b\Vert\ci{\text{BMO}(\nu)},
\end{equation}
where the supremum is taken over all cubes $Q\subseteq\R^n$, the constants $0<c,C<\infty$, depend on $n,p$ and $[\mu]\ci{A_p(\R^n)},[\lambda]\ci{A_p(\R^n)}$. In \cite{holmes-lacey-wick} it was also shown that $\text{BMO}(\nu)$ is the dual to a certain weighted Hardy space.

\textbf{Matrix-weighted BMO spaces.} Due to the lack of commutativity of matrices, one has to formulate all matrix-weighted BMO spaces directly in a two-weight fashion. Namely, let $1<p<\infty$ and consider two $d \times d$-matrix valued $A_p$ weights $U,V$ on $\R^n$. Let $B\in L^1\ti{loc}(\R^n;M_{d}(\C)),$ that is a locally integrable complex $d \times d$-matrix valued function on $\R^n$ in the sense that the matrix norm $|B(x)|$ is a locally integrable real valued function. We define the two matrix-weighted Bloom BMO norm $\Vert B\Vert\ci{\text{BMO}(U,V,p)}$ by
\begin{equation}
\label{matrix one-parameter two-weight BMO}
\Vert B\Vert\ci{\text{BMO}(U,V,p)}:=\sup_{Q}\left(\strokedint_{Q}|V(x)^{1/p}(B(x)-\La B\Ra\ci{Q})\cU^{-1}\ci{Q}|^{p}\mathd x\right)^{1/p},
\end{equation}
where the supremum is taken over all cubes $Q$ in $\R^n$ with sides parallel to the coordinate hyperplanes, and where $\cU_Q$ denotes the reducing operator of $U$ over the cube $Q$ taken with respect to exponent $p$ (see Section~\ref{subsec:ReducingOperators} for the corresponding definition and a brief exposition). This is the direct matrix-valued analog of the middle term of \eqref{holmes-lacey-wick John-Nirenberg}. As explained in \cite{isralowitz-pott-treil}, the direct matrix-valued analog of \eqref{weighted BMO} is
\begin{equation}
\label{matrix one-parameter ``one-weight'' BMO}
\Vert B\Vert\ci{\wt{\text{BMO}}(U,V,p)}:=\sup_{Q}\strokedint_{Q}|\cV\ci{Q}(B(x)-\La B\Ra\ci{Q})\cU^{-1}\ci{Q}|\mathd x,
\end{equation}
where the supremum is taken over all cubes $Q$ in $\R^n$ with sides parallel to the coordinate hyperplanes, and all reducing operators are taken with respect to exponent $p$. These matrix-weighted BMO spaces and variants of them were studied extensively in \cite{isralowitz duality}, \cite{isralowitz-kwon-pott}, culminating in \cite{isralowitz-pott-treil} which gave a complete theory of such spaces. In particular, \cite{isralowitz-pott-treil} established a matrix-valued analog of \eqref{holmes-lacey-wick John-Nirenberg},
\begin{equation}
\label{matrix valued BMO equivalence}
\Vert B\Vert\ci{\wt{\text{BMO}}(U,V,p)}\sim\Vert B\Vert\ci{\text{BMO}(U,V,p)}\sim\Vert B^{\ast}\Vert\ci{\text{BMO}(V',U',p')},
\end{equation}
where all implied constants depend only on $n,d,p,[U]\ci{A_p(\R^n)}$ and $[V]\ci{A_p(\R^n)}$. Moreover, work \cite{isralowitz duality} established in addition an $H^1$-BMO duality type result in the matrix-weighted setting for $p=2$.

Isralowitz--Pott--Treil \cite{isralowitz-pott-treil} showed that $\text{BMO}(U,V,p)$ is the correct space to characterize boundedness between matrix-weighted spaces of commutators of the form $[T,B]$, where $T$ is a \cz operator on $\R^n$ (see Subsection \ref{subsection: vector valued extensions} for the precise definition of $[T,B]$ in this context). They used a clever tensorization trick that allows one to reduce $d\times d$-matrix weighted estimates for the commutator $[T,B]$ to \emph{higher-dimensional} $(2d)\times(2d)$-matrix weighted estimates for the operator $T$. We refer to \cite{isralowitz-pott-treil} for details. It is worth noting that the results of \cite{isralowitz-pott-treil} were extended a little later by Cardenas--Isralowitz \cite{cardenas-isralowitz} to fractional matrix-weighted BMO spaces and commutators of fractional integral operators.

\textbf{Biparameter BMO spaces.} The study of \emph{biparameter} BMO spaces (i.~e.~invariant under rescaling each variable separately) was initiated by S.~Y.~A.~Chang \cite{chang} and R.~Fefferman \cite{fefferman}. These works investigated extensively the biparameter product BMO space $\text{BMO}(\R\times\R)$ on the product space $\R\times\R$, which consists of all locally integrable functions $b$ on $\R^2$ such that
\begin{equation*}
\Vert b\Vert\ci{\text{BMO}(\R\times\R)}:=\sup_{\Omega}\bigg(\frac{1}{|\Omega|}\sum_{\substack{R\in\bfD\\R\subseteq\Omega}}\left|(b,w\ci{R})\right|^2\bigg)^{1/2}<\infty,
\end{equation*}
where the supremum is taken over all (say) bounded Borel subsets $\Omega$ of $\R^2$ with nonzero measure,
where $\bfD$ stands for the family of all dyadic rectangles of $\R^2,$
and $(w\ci{R})\ci{R\in\bfD}$ is some (mildly regular) wavelet system adapted to $\bfD$. Here and below we denote
\begin{equation*}
(b,w\ci{R}):=\int_{\R^2}b(x)w\ci{R}(x)\dd x.
\end{equation*}
These works by Chang and Fefferman established in particular that $\text{BMO}(\R\times\R)$ is the dual to the biparameter Hardy space.

One can also define a different (nonequivalent) biparameter BMO space
by an oscillation expression analogous to that used in classical BMO.
To this end, let $b$ be a locally integrable function on $\R^{n}\times\R^{m}$. We say that $b$ belongs to the little BMO space $\mathrm{bmo}(\R^n\times\R^m)$ if
\begin{equation*}
\Vert b\Vert\ci{\text{bmo}(\R^{n}\times\R^{m})}:=\sup_{R}\frac{1}{|R|}\int_{R}|b(x)-\La b\Ra\ci{R}|\dd x<\infty,
\end{equation*}
where the supremum ranges over all rectangles $R$ in $\R^{n}\times\R^{m}$ with sides parallel to the coordinate planes but of arbitrary eccentricity, and where $\langle b \rangle_R$ denotes the average of $b$ over the rectangle $R$. Given now a weight $\nu$ on $\R^{n}\times\R^{m}$ and a locally integrable function $b$ on $\R^{n}\times\R^{m}$, I.~Holmes, S.~Petermichl and B.~D.~Wick \cite{holmes-petermichl-wick} introduced and studied the weighted little BMO space $\text{bmo}(\nu)$ which consists of all locally integrable functions $b$ on $\R^{n}\times\R^{m}$ such that
\begin{equation*}
\Vert b\Vert\ci{\text{bmo}(\nu)}:=\sup_{R}\frac{1}{\nu(R)}\int_{R}|b(x)-\La b\Ra\ci{R}|\dd x<\infty,
\end{equation*}
where the supremum ranges over all rectangles $R$ in $\R^{n}\times\R^{m}$ with sides parallel to the coordinate planes but of arbitrary eccentricity. In \cite{holmes-petermichl-wick} it is shown (extending an earlier result of S.~H.~Ferguson and C.~Sadosky \cite{ferguson-sadosky} for the tensor product of two Hilbert transforms in the unweighted case) that if $1<p<\infty$, $\mu,\lambda$ are biparameter $A_p$ weights on $\R^{n}\times\R^{m}$ and $\nu=\mu^{1/p}\lambda^{-1/p}$, then for any Journ\'e operator $T$ on $\R^{n}\times\R^{m}$ one has the inequality
\begin{equation}
\label{hpw-upper}
\Vert [T,b]\Vert\ci{L^{p}(\mu)\rightarrow L^{p}(\lambda)}\leq C\Vert b\Vert\ci{\text{bmo}(\nu)},
\end{equation}
where the constant $0<C<\infty$ depends on $n,m,p,T$ and $[\mu]\ci{A_p(\R^{n}\times\R^{m})},[\lambda]\ci{A_p(\R^{n}\times\R^{m})}$. Moreover, \cite{holmes-petermichl-wick} establishes that, with the same conditions on $p$, $\mu$ and $\lambda$ defined on $\R^{n}\times\R^{n}$ and if $R^{(1)}, \ldots, R^{(n)}$ are the Riesz transforms on $\R^n$, then
\begin{equation*}
\sum_{i,j=1}^{n}\Vert[R^{(i)}\otimes R^{(j)},b]\Vert\ci{L^{p}(\mu)\rightarrow L^{p}(\lambda)}\geq c\Vert b\Vert\ci{\text{bmo}(\nu)},
\end{equation*}
where the constant $0<c<\infty$ depends only on $n,p$ and $[\mu]\ci{A_p(\R^n\times\R^n)},[\lambda]\ci{A_p(\R^n\times\R^n)}$. Holmes--Petermichl--Wick \cite{holmes-petermichl-wick} provided through an application of the Lebesgue differentiation theorem a characterization of weighted little BMO in terms of usual (one-parameter) BMO, namely that
\begin{equation}
\label{holmes-petermichl-wick little BMO}
\Vert b\Vert\ci{\text{bmo}(\nu)} \sim \max\bigg(\esssup_{x_2\in\R^{m}}\Vert b(\fdot,x_2)\Vert\ci{\text{BMO}(\nu(\fdot,x_2))},~\esssup_{x_1\in\R^{n}}\Vert b(x_1,\fdot)\Vert\ci{\text{BMO}(\nu(x_1,\fdot))}\bigg)
\end{equation}
with the comparability constants depending only on $n,m$ and $[\nu]\ci{A_2(\R^{n}\times\R^{m})}$, allowing them to deduce weighted John--Nirenberg type inequalities for weighted little BMO analogous to estimate \eqref{holmes-lacey-wick John-Nirenberg}.

\subsection*{Structure of the article}
This article is structured in the following way.
In Section~\ref{s:background} we review the general background, definitions and properties of the objects that are studied here.
Section~\ref{s:MatrixWeightedLittleBMO} is devoted to the properties of the little BMO spaces and the proof of their equivalence
stated in Theorem~\ref{thm:equivalences little bmo}.
Finally, in Sections~\ref{s: lower matrix little BMO bounds} and~\ref{s: upper matrix little BMO bounds} we show, respectively,
Theorems~\ref{thm:MatrixBMOLowerBound} and~\ref{thm:MatrixBMOUpperBound}.

\subsection*{Acknowledgments}
The authors would like to thank Stefanie Petermichl for suggesting the problem treated here.

\section{Background}
\label{s:background}

In this section we fix some important definitions and notation, and we recall known facts about matrix-valued weights that we will need in the sequel.

\subsection{Tensor products of operators}

If $f$ is any function on a product space $\R^n\times\R^m$, then for all $x_1\in\R^n$, we denote by $f_{x_1}$ the function on $\R^m$ given by
\begin{equation*}
f_{x_1}(x_2):=f(x_1,x_2).
\end{equation*}
Moreover, for all $x_2\in\R^m$, we denote by $f_{x_2}$ the function on $\R^n$ given by
\begin{equation*}
f_{x_2}(x_1):=f(x_1,x_2).
\end{equation*}
We will be often denoting $f_{x_1},f_{x_2}$ by $f(x_1,\fdot),f(\fdot,x_2)$ respectively.

Let $T_1$ be an operator acting on (suitable) $\C^d$-valued functions $f$ on $\R^n$ such that $T_1f$ is again a $\C^d$-valued function on $\R^n$, and let $T_2$ be an operator acting on (suitable) $\C^d$-valued functions $f$ on $\R^m$ such that $T_2f$ is again a $\C^d$-valued function on $\R^m$. The tensor product $T_1\otimes T_2$ is defined as the operator acting on (suitable) $\C^d$-valued functions $f$ on the product space $\R^n\times\R^m$ by
\begin{equation*}
(T_1\otimes T_2)f(x_1,x_2):=T_2(g(x_1,\fdot))(x_2),\qquad g(x_1,y_2):=T_1(f(\fdot,y_2))(x_1),
\end{equation*}
for $x_1\in\R^n$ and $x_2,y_2\in\R^m$.

\subsection{Vector-valued extensions of linear operators}
\label{subsection: vector valued extensions}

Let $T$ be a linear operator acting on (suitable) scalar-valued (i.e.~$\C$-valued) functions $f$ on $\R^n$, in such a way that $Tf$ is also a scalar-valued function on $\R^n$. Then, we define the (canonical) $d$-vector valued extension of $T$ acting on (suitable) $\C^d$-valued functions on $\R^n$ as follows. Given any (suitable) $\C^d$-valued function $f$ on $\R^n$, we set
\begin{equation*}
Tf:=(Tf_1,\ldots,Tf_d),
\end{equation*}
where $f=(f_1,\ldots,f_d)$, i.e. $f_i$ is the scalar-valued function on $\R^n$ obtained as the $i$-th coordinate projection of $f$, for all $i=1,\ldots,d$; observe that $Tf$ is itself a 
$\C^d$-valued function on $\R^n$. Notice that abusing notation, we have denoted $T$ and its vector-valued extension by the same symbol; we will always be doing so in the sequel.

Given a linear operator $T$ acting on $\R^d$-valued functions on $\R^n$ and a $d\times d$-matrix valued function $B$ on $\R^n$, we can consider the commutator $[T,B]$ acting on $\R^d$-valued functions on $\R^n$ by
\begin{equation*}
[T,B](f):=T(Bf)-B T(f).
\end{equation*}
 
\subsection{Dyadic grids} \label{s:dyadicgrid} For definiteness, in what follows intervals in $\R$ will always be assumed to be left-closed, right-open and bounded. A cube in $\R^n$ will be a set of the form $Q=I_1\times\ldots\times I_n$, where $I_k,~k=1,\ldots,n$ are intervals in $\R$ of the same length; we denote by $\ell(Q):=|I_1|$ the sidelength of $Q$. A rectangle in $\R^n\times\R^m$ (with sides parallel to the coordinate axes) will be a set of the form $R=R_1\times R_2$, where $R_1$ is a cube in $\R^n$ and $R_2$ is a cube in $\R^m$.

A collection $\cD$ of intervals in $\R$ will be said to be a \emph{dyadic grid} in $\R$ if one can write $\cD=\bigcup_{k\in\Z}\cD_{k}$, such that the following hold:
\begin{itemize}
\item for all $k\in\Z$, $\cD_k$ forms a partition of $\R$, and all intervals in $\cD_k$ have the same length $2^{-k}$

\item for all $k\in\Z$, every $J\in\cD_{k}$ can be written as a union of exactly $2$ intervals in $\cD_{k+1}$.
\end{itemize}
We say that $\cD$ is the \emph{standard dyadic grid} in $\R$ if
\begin{equation*}
\cD:=\lbrace [m 2^{k},(m+1) 2^{k}):~k,m\in\Z\rbrace.
\end{equation*}

A collection $\cD$ of cubes in $\R^n$ will be said to be a \emph{dyadic grid} in $\R^n$ if for some dyadic grids $\cD^1,\ldots,\cD^n$ in $\R$ one can write $\cD=\bigcup_{k\in\Z}\cD_{k}$, where
\begin{equation*}
\cD_{k}=\lbrace I_1\times\ldots\times I_{n}:~ I_{i}\in\cD^i,~|I_{i}|=2^{-k},~i=1,\ldots,n\rbrace.
\end{equation*}
We say that $\cD$ is the \emph{standard dyadic grid} in $\R^n$ if
\begin{equation*}
\cD:=\lbrace [m_1 2^{k},(m_1+1) 2^{k}) \times \dots \times [m_n 2^{k},(m_n+1) 2^{k}):~k,m_1,\ldots,m_n\in\Z\rbrace.
\end{equation*}
If $\cD$ is a dyadic grid in $\R^n$, then we denote
\begin{equation*}
\ch_{i}(Q):=\lbrace K\in\cD:~K\subseteq Q,~|K|=2^{-in}|Q|\rbrace,\qquad Q\in\cD,~i=0,1,2,\ldots.
\end{equation*}

A collection $\bfD$ is said to be a \emph{product dyadic grid} in $\R^n\times\R^m$ if for a dyadic grid $\cD^1$ in $\R^n$ and a dyadic grid $\cD^2$ in $\R^m$ we have
\begin{equation*}
\bfD:=\lbrace R_1\times R_2:~R_i\in\cD_i,~i=1,2\rbrace,
\end{equation*}
and in this case we write (slightly abusing the notation) $\bfD=\cD^1\times\cD^2$.

If $\bfD$ is a product dyadic grid in $\R^n\times\R^m$, then we denote
\begin{equation*}
\ch_{i}(R):=\lbrace Q_1\times Q_2:~Q_i\in\ch_{i_{j}}(R_{j}),~j=1,2\rbrace,\qquad R\in\bfD,~i=(i_1,i_2),~i_1,i_2=0,1,2,\ldots.
\end{equation*}
We say that $R$ is the $i$-th ancestor of $P$ in $\bfD$ if $P\in\ch_{i}(R)$.

\subsection{Expressing matrix norms in terms of the norms of the columns}

In the sequel we denote by $(e_1,\ldots,e_d)$ the standard basis of $\C^d$. We will be often using the fact that
\begin{equation}
\label{equivalence_matrix_norm_columns}
|A|\sim_{d}\sum_{k=1}^{d}|Ae_k|,\quad\forall A\in M_{d}(\C),
\end{equation}
without explicitly mentioning it. Note also that for all $0<p<\infty$ and for nonnegative numbers $x_1,\ldots,x_d$ we have the estimate
\begin{equation*}
\min(1,d^{p-1})\sum_{i=1}^{d}x_i^{p}\leq\left(\sum_{i=1}^{d}x_i\right)^{p}\leq\max(1,d^{p-1})\sum_{i=1}^{d}x_i^{p}.
\end{equation*}
In particular, if $F_k:\R^n\rightarrow M_d(\C)$ is a sequence of Lebesgue-measurable functions and $0<p,q,r<\infty$, then we have
\begin{align}
\label{interchange Lebesgue norms and d sum}
&\left(\int_{\R^n}\left(\sum_{k=1}^{\infty}|F_k(x)|^r\right)^{p}\mathd x\right)^{q}
\sim_{d,p,q,r}\left(\int_{\R^n}\left(\sum_{k=1}^{\infty}\left(\sum_{i=1}^{d}|F_k(x)e_i|\right)^r\right)^{p}\mathd x\right)^{q}\\
\nonumber&\sim_{d,p,q,r}\left(\int_{\R^n}\left(\sum_{k=1}^{\infty}\sum_{i=1}^{d}|F_k(x)e_i|^r\right)^{p}\mathd x\right)^{q}
\sim_{d,p,q}\left(\int_{\R^n}\sum_{i=1}^{d}\left(\sum_{k=1}^{\infty}|F_k(x)e_i|^r\right)^{p}\mathd x\right)^{q}\\
\nonumber&\sim_{d,q}\sum_{i=1}^{d}\left(\int_{\R^n}\left(\sum_{k=1}^{\infty}|F_k(x)e_i|^r\right)^{p}\mathd x\right)^{q}.
\end{align}

\subsection{Matrix-valued weighted Lebesgue spaces}

A function $W$ on $\R^n$ is said to be a \emph{$d\times d$-matrix valued weight} if it is a locally integrable $M_{d}(\C)$-valued function such that $W(x)$ is a positive-definite matrix for a.e.~$x\in\R^n$. Given a $d\times d$-matrix valued weight $W$ on $\R^n$ and $1<p<\infty$, we define
\begin{equation*}
\Vert f\Vert\ci{L^{p}(W)}:=\left(\int_{\R^n}|W(x)^{1/p}f(x)|^{p}dx\right)^{1/p},
\end{equation*}
for all $\C^d$-valued measurable functions $f$ on $\R^n$.

\subsection{Reducing operators}
\label{subsec:ReducingOperators}
Let $1<p<\infty$. Let $E$ be a bounded measurable subset of $\R^n$ of nonzero measure. Let $W$ be a $M_{d}(\C)$-valued function on $E$ that is integrable over $E$ (meaning that $\int_{E}|W(x)|\mathd x<\infty$) and that takes a.e.~values in the set of positive-definite $d\times d$-matrices. It is proved in \cite[Proposition 1.2]{goldberg} (in the generality of abstract norms on $\C^d$) as an application of the standard properties of John ellipsoids that there exists a (not necessarily unique) positive-definite matrix $\cW\ci{E}\in M_{d}(\C)$, called \emph{reducing operator} of $W$ over $E$ with respect to the exponent $p$, such that
\begin{equation}
\label{reducing_matrices_definition}
\left(\strokedint_{E}|W(x)^{1/p}e|^{p}\mathd x\right)^{1/p}\leq |\cW\ci{E} e|\leq\sqrt{d}\left(\strokedint_{E}|W(x)^{1/p}e|^{p}\mathd x\right)^{1/p},\qquad\forall e\in\C^d.
\end{equation}
If $d=1$, i.e. $W$ is scalar-valued, then one can clearly take (and we will always be taking) $\cW\ci{E}:=(W)\ci{E}^{1/p}$. Moreover, if $p=2$, then
\begin{equation*}
\strokedint_{E}|W(x)^{1/2}e|^{2}\mathd x=\strokedint_{E}\La W(x)e,e\Ra \mathd x=\La (W)\ci{E}e,e\Ra=|( W)\ci{E}^{1/2}e|^2,\qquad\forall e\in\C^d,
\end{equation*}
and thus in this special case one can take (and we will always be taking) $\cW\ci{E}:=(W)\ci{E}^{1/2}$.

Assume in addition now that the function $W' \coloneq W^{-1/(p-1)}$ is also integrable over $E$. Then we let $\cW'\ci{E}$ be the reducing matrix of $W'$ over $E$ corresponding to the exponent $p':=p/(p-1)$, so that
\begin{equation*}
|\cW'\ci{E} e|\sim_{d}\left(\strokedint_{E}|W'(x)^{1/p'}e|^{p'}\mathd x\right)^{1/p'}=\left(\strokedint_{E}|W(x)^{-1/p}e|^{p'}\mathd x\right)^{1/p'},\qquad\forall e\in\C^d.
\end{equation*}
Note that one can take (and we will always be taking) $\cW''\ci{E}=\cW\ci{E}$. Observe that
\begin{align*}
|\cW\ci{E}\cW\ci{E}'|\sim_{p,d}\left(\strokedint_{E}\left(\strokedint_{E}|W(x)^{1/p}W(y)^{-1/p}|^{p'}\mathd y \right)^{p/p'}\mathd x\right)^{1/p}.
\end{align*}
For a detailed exposition of reducing operators we refer for example to \cite{matrix journe}. Here we note just the following estimates that we will need in the sequel. A proof of part (1) can be found, for example, in \cite{matrix journe}. A proof of part (2) can be found in \cite{isralowitz-kwon-pott}.

\begin{lm}
\label{l: replace inverse by prime}
Let $W$ be a $M_{d}(\C)$-valued function on $E$, such that $W$ and $W'=W^{-1/(p-1)}$ are integrable over $E$ for some $1<p<\infty$. Set
\begin{align*}
&C\ci{E}:=\strokedint_{E}\left(\strokedint_{E}|W(x)^{1/p}W(y)^{-1/p}|^{p'}\mathd y \right)^{p/p'}\mathd x.
\end{align*}
We consider reducing operators of $W$ with respect to exponent $p$ and of $W'$ with respect to exponent $p'$.

\begin{enumerate}
\item[(1)] There holds
\begin{equation*}
|\cW\ci{E}^{-1} e|\leq |\cW\ci{E}'e|\lesssim_{p,d}C\ci{E}^{1/p}|\cW\ci{E}^{-1}e|,\qquad\forall e\in\C^d.
\end{equation*}

\item[(2)] There holds
\begin{equation*}
|\La W^{1/p}\Ra\ci{E}e|\leq|\cW\ci{E}e|\lesssim_{p,d}C\ci{E}^{d/p}|\La W^{1/p}\Ra\ci{E}e|,\qquad\forall e\in\C^d.
\end{equation*}
\end{enumerate}
\end{lm}

\subsubsection{Iterating reducing operators}

Let $E,F$ be measurable subsets of $\R^n,\R^m$ respectively with $0<|E|,|F|<\infty$. Let $1<p<\infty$. Let $W$ be a $M_{d}(\C)$-valued integrable function on $E\times F$ taking a.e.~values in the set of positive-definite $d\times d$-matrices, such that $W'=W^{-1/(p-1)}$ is also integrable over $E\times F$. For all $x_1\in E$, set $W_{x_1}(x_2):=W(x_1,x_2)$, $x_2\in F$. For a.e.~$x_1\in E$, denote by $\cW\ci{x_1,F}$ the reducing operator of $W_{x_1}$ over $F$ with respect to the exponent $p$. If $p=2$, then
\begin{equation*}
\cW\ci{x_1,F}:=(W(x_1,\fdot))\ci{F}^{1/2}.
\end{equation*}
and therefore the function $\cW\ci{x_1,F}$ is clearly measurable in $x_1$. In fact, it is not hard to see that for any $1<p<\infty$, one can choose the reducing operator $\cW\ci{x_1,F}$ in a way that is measurable in $x_1$. A proof of that can be found in the appendix of \cite{matrix journe}.

Set $W\ci{F}(x_1):=(\cW\ci{x_1,F})^{p}$, for a.e.~$x_1\in E$. Then, it is proved in \cite{matrix journe} that $W\ci{F}\in L^1(E;M_{d}(\C))$, and that
\begin{equation}
\label{iterate reducing operators}
|\cW\ci{F,E}e|\sim_{p,d}|\cW\ci{E\times F}e|,\qquad\forall e\in\C^d,
\end{equation}
where $\cW\ci{F,E}$ is the reducing operator of $W\ci{F}$ over $E$ with respect to the exponent $p$. Moreover, set $W'\ci{F}(x_1):=(W\ci{F}(x_1))^{-1/(p-1)}=\cW\ci{x_1,F}^{-p'}$, for a.e.~$x_1\in E$. It is proved in \cite{matrix journe} that $W\ci{F}'\in L^1\ti{loc}(E;M_{d}(\C))$. We denote by $\cW\ci{F,E}'$ the reducing operator of $W\ci{F}'$ over $E$ with respect to the exponent $p'$.

Of course, analogous facts hold if one ``iterates" the operation of taking reducing operators in the reverse order.

\subsection{Matrix \texorpdfstring{$A_p$}{Ap} weights}
\label{s:Ap}

\subsubsection{One-parameter matrix \texorpdfstring{$A_p$}{Ap} weights}

Let $W$ be a $d\times d$-matrix valued weight on $\R^n$, i.e. $W$ is a locally integrable $M_{d}(\C)$-valued function on $\R^n$ that takes a.e.~values in the set of positive-definite $d\times d$-matrices. We say that $W$ is a (one-parameter) $d\times d$-matrix valued $A_p$ weight if
\begin{equation*}
[W]\ci{A_p(\R^n)}:=\sup_{Q}\strokedint_{Q}\left(\strokedint_{Q}|W(x)^{1/p}W(y)^{-1/p}|^{p'}\mathd y \right)^{p/p'}\mathd x<\infty,
\end{equation*}
where the supremum is taken over all cubes $Q$ in $\R^{n}$. Note that if $W$ is a $d\times d$-matrix valued $A_p$ weight on $\R^n$, then $W':=W^{-1/(p-1)}$ is a $d\times d$-matrix valued $A_{p'}$ weight on $\R^n$ with $[W']\ci{A_{p'}(\R^n)}^{1/p'}\sim_{p,d}[W]\ci{A_p(\R^n)}^{1/p}$, and
\begin{equation*}
[W]\ci{A_p(\R^n)}\sim_{p,d}\sup_{Q}|\cW'\ci{Q}\cW\ci{Q}|^{p},
\end{equation*}
where the reducing matrices for $W$ correspond to exponent $p$, and those for $W'$ correspond to exponent $p'$.

If $\cD$ is any dyadic grid in $\R^n$, we define
\begin{equation*}
[W]\ci{A_p,\cD}:=\sup_{Q\in\cD}\strokedint_{Q}\left(\strokedint_{Q}|W(x)^{1/p}W(y)^{-1/p}|^{p'}\mathd y \right)^{p/p'}\mathd x,
\end{equation*}
and we say that $W$ is a (one-parameter) $d\times d$-matrix valued $\cD$-dyadic $A_p$ weight if $[W]\ci{A_p,\cD}<\infty$.

\subsubsection{Biparameter matrix \texorpdfstring{$A_p$}{Ap} weights}

Let $W$ be a $d\times d$-matrix valued weight on $\R^n\times\R^m$. We say that $W$ is a biparameter $d\times d$-matrix valued $A_p$ weight if
\begin{equation*}
[W]\ci{A_p(\R^n\times\R^m)}:=\sup_{R}\strokedint_{R}\left(\strokedint_{R}|W(x)^{1/p}W(y)^{-1/p}|^{p'}\mathd y \right)^{p/p'}\mathd x<\infty,
\end{equation*}
where the supremum is taken over all rectangles $R$ in $\R^{n}$ (with sides parallel to the coordinate axes). Note that if $W$ is a $d\times d$-matrix valued biparameter $A_p$ weight on $\R^n\times\R^m$, then $W':=W^{-1/(p-1)}$ is a $d\times d$-matrix valued biparameter $A_{p'}$ weight on $\R^n\times\R^m$ with $[W']\ci{A_{p'}(\R^n\times\R^m)}^{1/p'}\sim_{d,p}[W]\ci{A_p(\R^n\times\R^m)}^{1/p}$, and
\begin{equation*}
[W]\ci{A_p(\R^n\times\R^m)}\sim_{p,d}\sup_{R}|\cW'\ci{R}\cW\ci{R}|^{p},
\end{equation*}
where the reducing matrices for $W$ correspond to exponent $p$, and those for $W'$ correspond to exponent $p'$.

If $\bfD$ is any product dyadic grid in $\R^n\times\R^m$, we define
\begin{equation*}
[W]\ci{A_p,\bfD}:=\sup_{R\in\bfD}\strokedint_{R}\left(\strokedint_{R}|W(x)^{1/p}W(y)^{-1/p}|^{p'}\mathd y \right)^{p/p'}\mathd x;
\end{equation*}
we say that $W$ is a biparameter $d\times d$-matrix valued $\bfD$-dyadic $A_p$ weight if $[W]\ci{A_p,\bfD}<\infty$.

The following result, proved in \cite{matrix journe}, will play an important role below.

\begin{lm}
\label{l: two-weight-biparameter A_p implies uniform A_p in each coordinate}
Let $W$ be a $d\times d$-matrix valued weight on $\R^{n+m}$.
\begin{enumerate}

\item[(1)] Let $\bfD=\cD^1\times\cD^2$ be any product dyadic grid in $\R^n\times\R^m$. Then, there holds
\begin{equation*}
[W(\fdot,x_2)]\ci{A_p,\cD^1}\lesssim_{d,p}[W]\ci{A_p,\bfD},\qquad\text{for a.e.~} x_2\in \R^m,
\end{equation*}
\begin{equation*}
[W(x_1,\fdot)]\ci{A_p,\cD^2}\lesssim_{d,p}[W]\ci{A_p,\bfD},\qquad\text{for a.e.~} x_1\in \R^n.
\end{equation*}

\item[(2)] There holds
\begin{equation*}
[W(\fdot,x_2)]\ci{A_p(\R^n)}\lesssim_{d,p,n}[W]\ci{A_p(\R^n\times\R^m)},\qquad\text{for a.e.~} x_2\in \R^m,
\end{equation*}
\begin{equation*}
[W(x_1,\fdot)]\ci{A_p(\R^m)}\lesssim_{d,p,m}[W]\ci{A_p(\R^n\times\R^m)},\qquad\text{for a.e.~} x_1\in \R^n.
\end{equation*}

\end{enumerate}
\end{lm}

\section{Matrix-weighted little BMO}
\label{s:MatrixWeightedLittleBMO}

In this section we define and study the matrix-weighted little BMO space in two parameters. Emphasis is led on proving that several possible reasonable norms are actually equivalent. The development parallels that of \cite{holmes-petermichl-wick} and \cite{isralowitz-pott-treil}.

\subsection{Dyadic versions} Let $1<p<\infty$. Let $\bfD=\cD^1\times\cD^2$ be any product dyadic grid in $\R^n\times\R^m$. Let $U,V$ be biparameter $d\times d$-matrix valued $\bfD$-dyadic $A_p$ weights on $\R^n\times\R^m$. Let $B\in L^1\ti{loc}(\R^{n+m};M_{d}(\C))$. We define the two-matrix weighted $\bfD$-dyadic little BMO norm $\Vert B\Vert\ci{\text{bmo}\ci{\bfD}(U,V,p)}$ by
\begin{equation}
\label{matrix two-weight little BMO dyadic}
\Vert B\Vert\ci{\text{bmo}\ci{\bfD}(U,V,p)}:=\sup_{R\in\bfD}\left(\strokedint_{R}|V(x)^{1/p}(B(x)-\La B\Ra\ci{R})\cU\ci{R}^{-1}|^{p}\mathd x\right)^{1/p},
\end{equation}
where all reducing operators are taken with respect to exponent $p$. We also define
\begin{equation}
\label{matrix ``one-weight'' little BMO dyadic}
\Vert B\Vert\ci{\wt{\text{bmo}}\ci{\bfD}(U,V,p)}:=\sup_{R\in\bfD}\strokedint_{R}|\cV\ci{R}(B(x)-\La B\Ra\ci{R})\cU\ci{R}^{-1}|\mathd x,
\end{equation}
where all reducing operators are again taken with respect to exponent $p$. As in the one-parameter case \cite{isralowitz-pott-treil}, we will also need the variants
\begin{equation}
\label{matrix two-weight little BMO first pointwise dyadic}
\Vert B\Vert\ci{\text{bmo}^1\ci{\bfD}(U,V,p)}:=\sup_{R\in\bfD}\left(\strokedint_{R}\left(\strokedint_{R}|V(x)^{1/p}(B(x)-B(y))U(y)^{-1/p}|^{p'}\mathd y\right)^{p/p'}\mathd x\right)^{1/p}
\end{equation}
and
\begin{equation}
\label{matrix two-weight little BMO second pointwise dyadic}
\Vert B\Vert\ci{\text{bmo}^2\ci{\bfD}(U,V,p)}:=\sup_{R\in\bfD}\left(\strokedint_{R}\left(\strokedint_{R}|V(x)^{1/p}(B(x)-B(y))U(y)^{-1/p}|^{p}\mathd x\right)^{p'/p}\mathd y\right)^{1/p'}.
\end{equation}
Observe that $\Vert B\Vert\ci{\text{bmo}^2\ci{\bfD}(U,V,p)}=\Vert B^{\ast}\Vert\ci{\text{bmo}^1\ci{\bfD}(V',U',p')}$.

Our main goal is to prove the following

\begin{prop}
\label{prop:equivalencesdyadiclittleBMO}
Let $d,p,\bfD,U,V,B$ as above. Then, there holds
\begin{equation}
\label{mainBMOdyadicequivalence}
\Vert B\Vert\ci{\emph{bmo}\ci{\bfD}(U,V,p)}\sim\Vert B^{\ast}\Vert\ci{\emph{bmo}\ci{\bfD}(V',U',p')}\sim\Vert B\Vert\ci{\wt{\emph{bmo}}\ci{\bfD}(U,V,p)}\sim\Vert B^{\ast}\Vert\ci{\wt{\emph{bmo}}\ci{\bfD}(V',U',p')},
\end{equation}
and
\begin{equation}
\label{secondaryBMOdyadicequivalence}
\Vert B\Vert\ci{\emph{bmo}\ci{\bfD}(U,V,p)}\sim\Vert B\Vert\ci{\emph{bmo}^1\ci{\bfD}(U,V,p)}\sim\Vert B\Vert\ci{\emph{bmo}^2\ci{\bfD}(U,V,p)},
\end{equation}
where all implied constants depend only on $n,m,p,d, [U]\ci{A_p,\bfD}$ and $[V]\ci{A_p,\bfD}$
\end{prop}

The proof of Proposition \ref{prop:equivalencesdyadiclittleBMO} will be accomplished in several steps. Our main strategy consists in reducing its content to its one-parameter counterpart \cite[Lemma 4.7]{isralowitz-pott-treil}. That is, we prove first an analog of \cite[Proposition 4.2]{holmes-petermichl-wick}.

We fix the following pieces of notation. For all $x_1\in\R^n$, we set $U_{x_1}(x_2):=U(x_1,x_2)$, $x_2\in \R^m$. For a.e.~$x_1\in \R^n$, we denote by $\cU\ci{x_1,Q}$ the reducing operator of $U_{x_1}$ over any cube $Q$ in $\R^m$ with respect to the exponent $p$, and we define $U\ci{Q}(x_1):=\cU\ci{x_1,Q}^{p}$. If $P,Q$ are cubes in $\R^n,\R^m$ respectively, we denote by $\cU\ci{Q,P}$ the reducing operator of $U\ci{Q}$ over $P$ with respect to the exponent $p$, and by $\cU\ci{Q,P}'$ the reducing operator of $U\ci{Q}':=(U\ci{Q})^{-1/(p-1)}$ over $P$ with respect to the exponent $p'$. We fix similar pieces of notation for $V$, and also the obvious symmetric pieces of notation. We also denote $B_{x_1}(x_2)=B_{x_2}(x_1):=B(x_1,x_2)$ for $(x_1,x_2)\in\R^n\times\R^m$. Finally we write in general $x=(x_1,x_2)\in\R^n\times\R^m$.

\begin{lm}
\label{l: step1}
There holds
\begin{equation*}
\Vert B\Vert\ci{\emph{bmo}\ci{\bfD}(U,V,p)}\lesssim\max(\esssup_{x_2\in\R^m}\Vert B_{x_2}\Vert\ci{\emph{BMO}\ci{\cD^1}(U_{x_2},V_{x_2},p)},\esssup_{x_1\in\R^n}\Vert B_{x_1}\Vert\ci{\emph{BMO}\ci{\cD^2}(U_{x_1},V_{x_1},p)}),
\end{equation*}
where the implied constant depends only on $p,d$ and $[V]\ci{A_p,\bfD}$.
\end{lm}

\begin{proof}
The proof combines and adapts parts of the proofs of \cite[Propositions 4.2, 4.5]{holmes-petermichl-wick}. Set
\begin{equation*}
C_1:=\esssup_{x_2\in\R^m}\Vert B_{x_2}\Vert\ci{\text{BMO}\ci{\cD^1}(U_{x_2},V_{x_2},p)},\qquad
C_2:=\esssup_{x_1\in\R^n}\Vert B_{x_1}\Vert\ci{\text{BMO}\ci{\cD^2}(U_{x_1},V_{x_1},p)}.
\end{equation*}
Without loss of generality we may assume that $C_1,C_2<\infty$. Let $R=R_1\times R_2\in\bfD$ be arbitrary. Then, we have
\begin{align*}
\left(\strokedint_{R}|V(x)^{1/p}(B(x)-\La B\Ra\ci{R})\cU\ci{R}^{-1}|^{p}\mathd x\right)^{1/p}
&\leq\left(\strokedint_{R}|V(x)^{1/p}(B(x)-\La B_{x_1}\Ra\ci{R_2})\cU\ci{R}^{-1}|^{p}\mathd x\right)^{1/p}\\
&+\left(\strokedint_{R}|V(x)^{1/p}(\La B_{x_1}\Ra\ci{R_2}-\La B\Ra\ci{R})\cU\ci{R}^{-1}|^{p}\mathd x\right)^{1/p}.
\end{align*}
We observe that
\begin{align*}
&\strokedint_{R}|V(x)^{1/p}(B(x)-\La B_{x_1}\Ra\ci{R_2})\cU\ci{R}^{-1}|^{p}\mathd x\\
&\leq\strokedint_{R_1}\left(\strokedint_{R_2}|V(x_1,x_2)^{1/p}(B_{x_1}(x_2)-\La B_{x_1}\Ra\ci{R_2})\cU\ci{x_1,R_2}^{-1}|^{p}\cdot|\cU\ci{x_1,R_2}\cU\ci{R}^{-1}|^{p}\mathd x_2\right)\mathd x_1\\
&\leq C_2^{p}\strokedint_{R_1}|\cU\ci{x_1,R_2}\cU\ci{R}^{-1}|^{p}\mathd x_1\sim_{p,d}C_2^{p}|\cU\ci{R_2,R_1}\cU\ci{R}^{-1}|^{p}\sim_{p,d}C_2^{p}|\cU\ci{R}\cU\ci{R}^{-1}|^{p}=C_2^{p},
\end{align*}
where in the last $\sim_{p,d}$ we used \eqref{iterate reducing operators}. Moreover, we have
\begin{align*}
&\strokedint_{R}|V(x)^{1/p}(\La B_{x_1}\Ra\ci{R_2}-\La B\Ra\ci{R})\cU\ci{R}^{-1}|^{p}\mathd x\sim_{p,d}\strokedint_{R_1}|\cV\ci{x_1,R_2}(\La B_{x_1}\Ra\ci{R_2}-\La B\Ra\ci{R})\cU\ci{R}^{-1}|^{p}\mathd x_1\\
&\leq\strokedint_{R_1}\left(\strokedint_{R_2}|\cV\ci{x_1,R_2}(B(x_1,y_2)-\La B(\fdot,y_2)\Ra\ci{R_1})\cU\ci{R}^{-1}|\mathd y_2\right)^{p}\mathd x_1\\
&\leq\strokedint_{R_1}\left(\strokedint_{R_2}|\cV\ci{x_1,R_2}V(x_1,y_2)^{-1/p}|^{p'}\mathd y_2\right)^{p/p'}\\
&\cdot\left(\strokedint_{R_2}|V(x_1,y_2)^{1/p}(B_{y_2}(x_1)-\La B_{y_2}\Ra\ci{R_1})\cU\ci{R}^{-1}|^{p}\mathd y_2\right)\mathd x_1\\
&\sim_{p,d}\strokedint_{R_1}|\cV'\ci{x_1,R_2}\cV\ci{x_1,R_2}|^{p}\cdot\left(\strokedint_{R_2}|V(x_1,y_2)^{1/p}(B_{y_2}(x_1)-\La B_{y_2}\Ra\ci{R_1})\cU\ci{R}^{-1}|^{p}\mathd y_2\right)\mathd x_1\\
&\lesssim_{p,d}[V]\ci{A_p,\bfD}\strokedint_{R_2}|\cU\ci{y_2,R_1}\cU^{-1}\ci{R}|^{p}\left(\strokedint_{R_1}|V(x_1,y_2)^{1/p}(B_{y_2}(x_1)-\La B_{y_2}\Ra\ci{R_1})\cU\ci{y_2,R_1}^{-1}|^{p}\mathd x_1\right)\mathd y_2\\
&\leq C_1^{p}[V]\ci{A_p,\bfD}\strokedint_{R_2}|\cU\ci{y_2,R_1}\cU^{-1}\ci{R}|^{p}\mathd y_2\sim_{p,d}C_1^{p}[V]\ci{A_p,\bfD}|\cU\ci{R_2,R_1}\cU^{-1}\ci{R}|^{p}\\
&\sim_{p,d}C_1^{p}[V]\ci{A_p,\bfD}|\cU\ci{R}\cU^{-1}\ci{R}|^{p}=C_1^{p}[V]\ci{A_p,\bfD},
\end{align*}
where in the first $\lesssim_{p,d}$ we used Lemma \ref{l: two-weight-biparameter A_p implies uniform A_p in each coordinate}, and in the last $\sim_{p,d}$ we used \eqref{iterate reducing operators}, concluding the proof.
\end{proof}

\begin{lm}
\label{l: step2}
There holds
\begin{equation*}
\Vert B\Vert\ci{\wt{\emph{bmo}}\ci{\bfD}(U,V,p)}\lesssim_{p,d}[V]\ci{A_p,\bfD}^{1/p}\Vert B\Vert\ci{\emph{bmo}\ci{\bfD}(U,V,p)}.
\end{equation*}
\end{lm}

\begin{proof}
In the one-parameter case this is proved in \cite[Lemma 4.4]{isralowitz-pott-treil}, and the same proof works without any changes in the present two-parameter setting. For reasons of completeness, we repeat the short argument here. For all $R\in\bfD$, we have
\begin{align*}
&\strokedint_{R}|\cV\ci{R}(B(x)-\La B\Ra\ci{R})\cU^{-1}\ci{R}|\mathd x\leq
\strokedint_{R}|\cV\ci{R}V(x)^{-1/p}|\cdot|V(x)^{1/p}(B(x)-\La B\Ra\ci{R})\cU^{-1}\ci{R}|\mathd x\\
&\leq\left(\strokedint_{R}|\cV\ci{R}V(x)^{-1/p}|^{p'}\mathd x\right)^{1/p'}\left(\strokedint_{R}|V(x)^{1/p}(B(x)-\La B\Ra\ci{R})\cU^{-1}\ci{R}|^{p}\mathd x\right)^{1/p}\\
&\sim_{p,d}|\cV'\ci{R}\cV\ci{R}|\left(\strokedint_{R}|V(x)^{1/p}(B(x)-\La B\Ra\ci{R})\cU^{-1}\ci{R}|^{p}\mathd x\right)^{1/p}\\
&\lesssim_{p,d}[V]\ci{A_p,\bfD}^{1/p}\left(\strokedint_{R}|V(x)^{1/p}(B(x)-\La B\Ra\ci{R})\cU^{-1}\ci{R}|^{p}\mathd x\right)^{1/p}.
\end{align*}
\end{proof}

\begin{lm}
\label{l: step3}
There holds
\begin{equation*}
\Vert B\Vert\ci{\wt{\emph{bmo}}\ci{\bfD}(U,V,p)}\gtrsim_{p,d}\max(\esssup_{x_2\in\R^m}\Vert B_{x_2}\Vert\ci{\wt{\emph{BMO}}\ci{\cD^1}(U_{x_2},V_{x_2},p)},\esssup_{x_1\in\R^n}\Vert B_{x_1}\Vert\ci{\wt{\emph{BMO}}\ci{\cD^2}(U_{x_1},V_{x_1},p)}).
\end{equation*}
\end{lm}

To prove Lemma \ref{l: step3} we need the following technical lemma.

\begin{lm}
\label{l: from rational to weak Lebesgue differentiation}
Let $F\in L^1\ti{loc}(\R^{n+m};M_{d}(\C))$ be arbitrary. Let $x_2\in\R^m$ such that $F(\fdot,x_2)\in L^1\ti{loc}(\R^n;M_{d}(\C))$, and let $(Q_k)^{\infty}_{k=1}$ be a sequence of cubes in $\R^m$ shrinking to $x_2$. Assume that for \emph{some} cube $P$ in $\R^n$ (possibly depending on $F$), there holds
\begin{equation*}
\strokedint_{P}|AF(x_1,x_2)B|\mathd x_1=\lim_{k\rightarrow\infty}\strokedint_{P\times Q_k}|AF(x_1,y_2)B|\mathd x_1\mathd y_2,\qquad\forall A,B\in M_{d}(\Q+i\Q).
\end{equation*}
Then, there holds
\begin{equation*}
\strokedint_{P}|AF(x_1,x_2)B|\mathd x_1\leq\liminf_{k\rightarrow\infty}\strokedint_{P\times Q_k}|AF(x_1,y_2)B|\mathd x_1\mathd y_2,\qquad\forall A,B\in M_{d}(\C).
\end{equation*}
\end{lm}

\begin{proof}
Let $A,B\in M_{d}(\C)$ be arbitrary. Pick sequences $(A_{\ell})^{\infty}_{\ell=1},(B_{\ell})^{\infty}_{\ell=1}$ in $M_{d}(\Q+i\Q)$ such that $\lim_{\ell\rightarrow\infty}A_{\ell}=A$, $\lim_{m\rightarrow\infty}B_{\ell}=B$. Then, since by assumption $F(\fdot,x_2)\in L^1\ti{loc}(\R^n;M_d(\C))$ by the Theorem of Dominated Convergence we deduce
\begin{equation*}
\strokedint_{P}|AF(x_1,x_2)B|\mathd x_1=\lim_{\ell\rightarrow\infty}\strokedint_{P}|A_{\ell}F(x_1,x_2)B_{\ell}|\mathd x_1.
\end{equation*}
We observe now that for all $k,\ell=1,2,\ldots$ we have
\begin{align*}
&\strokedint_{P\times Q_k}|A_{\ell}F(x_1,y_2)B_{\ell}|\mathd x_1\mathd y_2\leq\strokedint_{P\times Q_k}|AF(x_1,y_2)B|\mathd x_1\mathd y_2\\
&+\strokedint_{P\times Q_k}|AF(x_1,y_2)(B_{\ell}-B)|\mathd x_1\mathd y_2+\strokedint_{P\times Q_k}|(A_{\ell}-A)F(x_1,y_2)B|\mathd x_1\mathd y_2\\
&+\strokedint_{P\times Q_k}|(A_{\ell}-A)F(x_1,y_2)(B_{\ell}-B)|\mathd x_1\mathd y_2\\
&\leq\strokedint_{P\times Q_k}|AF(x_1,y_2)B|\mathd x_1\mathd y_2+|A|\cdot|B_{\ell}-B|\cdot\strokedint_{P\times Q_k}|F(x_1,y_2)|\mathd x_1\mathd y_2\\
&+|A_{\ell}-A|\cdot|B|\cdot\strokedint_{P\times Q_k}|F(x_1,y_2)|\mathd x_1\mathd y_2+|A_{\ell}-A|\cdot|B_{\ell}-B|\cdot\strokedint_{P\times Q_k}|F(x_1,y_2)|\mathd x_1\mathd y_2.
\end{align*}
Taking for every fixed $\ell=1,2,\ldots$ liminf as $k\rightarrow\infty$ and using the assumption we obtain
\begin{align*}
&\strokedint_{P}|A_{\ell}F(x_1,x_2)B_{\ell}|\mathd x_1\leq\liminf_{k\rightarrow\infty}\strokedint_{P\times Q_k}|AF(x_1,y_2)B|\mathd x_1\mathd y_2\\
&+(|A|\cdot|B_{\ell}-B|+|A_{\ell}-A|\cdot|B|+|A_{\ell}-A|\cdot|B_{\ell}-B|)\cdot\strokedint_{P}|F(x_1,x_2)|\mathd x_1.
\end{align*}
The desired result follows then after taking the limit as $\ell\rightarrow\infty$ in the last display.
\end{proof}

\begin{proof}[Proof (of Lemma \ref{l: step3})]
Broadly speaking, the proof we give is an adaptation of the first part of the proof of \cite[Proposition 4.2]{holmes-petermichl-wick}.

Without loss of generality we may assume that $\Vert B\Vert\ci{\wt{\text{bmo}}\ci{\bfD}(U,V,p)}<\infty$. We show only that
\begin{equation*}
\Vert B\Vert\ci{\wt{\text{bmo}}\ci{\bfD}(U,V,p)}\gtrsim_{p,d}\esssup_{x_2\in\R^m}\Vert B_{x_2}\Vert\ci{\wt{\text{BMO}}\ci{\cD^1}(U_{x_2},V_{x_2},p)},
\end{equation*}
the other estimate being symmetric.

We will rely on the ideas from the proof of \cite[Lemma 3.6]{matrix journe}. First of all, note that there exists a measurable subset $A$ of $\R^m$ such that $|\R^m\setminus A|=0$ and
\begin{align*}
U(\fdot,x_2),V(\fdot,x_2),B(\fdot,x_2)\in L^1\ti{loc}(\R^n;M_{d}(\C)),\qquad\forall x_2\in A.
\end{align*}
Then, for all $P\in\cD^1$, we can consider the a.e.~defined function $F\ci{P}$ on $\R^m$ given by
\begin{equation*}
F\ci{P}(x_2):=\strokedint_{P}|B(x_1,x_2)|\mathd x_1,\qquad\forall x_2\in A.
\end{equation*}
Clearly, $F\ci{P}\in L^1\ti{loc}(\R^m;M_{d}(\C))$, therefore we can consider the set $A\ci{P}$ of all Lebesgue points of $F\ci{P}$, for all $P\in\cD^1$. Set
\begin{equation*}
C:=A\cap\left(\bigcap_{P\in\cD^1}A\ci{P}\right).
\end{equation*}
Next, for all $S,T\in M_{d}(\Q+i\Q)$ and for all $P\in\cD^1$ consider the a.e.~on $\R^m$ defined locally integrable function
\begin{equation*}
H\ci{P,S,T}(y_2):=\strokedint_{P}|S(B(x_1,y_2)-\La B(\fdot,y_2)\Ra\ci{P})T|\mathd x_1,\qquad y_2\in C,
\end{equation*}
and let $C\ci{P,S,T}$ be the set of its Lebesgue points. Set
\begin{equation*}
C':=C\cap\bigg(\bigcap_{\substack{P\in\cD^1\\S,T\in M_{d}(\Q+i\Q)}}C\ci{P,S,T}\bigg),
\end{equation*}
so that $C'$ is a measurable subset of $\R^m$ with $|\R^m\setminus C'|=0$. Finally, as in the proof of \cite[Lemma 3.6]{matrix journe}, we have that there exists a measurable subset $E$ of $\R^m$ with $E\subseteq C'$ and $|\R^m\setminus E|=0$, such that the following holds: for each $x_2\in E$ and for each $P\in\cD^1$, there exist \emph{some} sequence $(Q_k)^{\infty}_{k=1}$ in $\cD^2$ shrinking to $x_2$ and \emph{some} self-adjoint matrices $\cU\ci{P,2},\cV\ci{P,2}$, such that
\begin{equation*}
\lim_{k\rightarrow\infty}\cU\ci{P\times Q_k}=\cU\ci{P,2},\qquad\lim_{k\rightarrow\infty}\cV\ci{P\times Q_k}=\cV\ci{P,2},
\end{equation*}
and
\begin{equation*}
|\cU\ci{P,2}e|\sim_{p,d}|\cU\ci{x_2,P}e|,\qquad|\cV\ci{P,2}e|\sim_{p,d}|\cV\ci{x_2,P}e|
\end{equation*}
for all $e\in\C^d$; in particular, note that the matrices $\cU\ci{P,2},\cV\ci{P,2}$ are invertible, because the matrices $\cU\ci{x_2,P},\cV\ci{x_2,P}$ are invertible, and moreover, by \cite[Lemma 3.1]{matrix journe} we deduce
\begin{equation*}
|\cU\ci{P,2}^{-1}e|\sim_{p,d}|\cU\ci{x_2,P}^{-1}e|,\qquad|\cV\ci{P,2}^{-1}e|\sim_{p,d}|\cV\ci{x_2,P}^{-1}e|,\qquad\forall e\in\C^d.
\end{equation*}

Pick now $x_2\in e$. We will prove that
\begin{equation*}
\Vert B_{x_2}\Vert\ci{\wt{\text{BMO}}\ci{\cD^1}(U_{x_2},V_{x_2},p)}\lesssim_{p,d}\Vert B\Vert\ci{\wt{\text{bmo}}\ci{\bfD}(U,V,p)}.
\end{equation*}
Let $P\in\cD^1$ be arbitrary. Pick a sequence $(Q_{n})_{n=1}^{\infty}$ in $\cD^2$ shrinking to $x_2$ as above. Consider the a.e.~defined function $f\ci{P}$ on $\R^m$ given by
\begin{equation*}
f\ci{P}(y_2):=\strokedint_{P}|\cV\ci{x_2,P}(B(x_1,y_2)-\La B(\fdot,y_2)\Ra\ci{P})\cU\ci{x_2,P}^{-1}|\mathd x_1,\qquad\forall y_2\in e.
\end{equation*}
It is obvious that $f\ci{P}$ is locally integrable. Our goal now is to show that $f\ci{P}(x_2)\lesssim_{p,d}\Vert B\Vert\ci{\wt{\text{bmo}}\ci{\bfD}(U,V,p)}$. Note that $x_2\in C'$, so by the construction of the set $C'$ coupled with Lemma \ref{l: from rational to weak Lebesgue differentiation} we deduce
\begin{equation*}
f\ci{P}(x_2)\leq\liminf_{k\rightarrow\infty}\La f\ci{P}\Ra\ci{Q_k}.
\end{equation*}
Next, for all $k=1,2,\ldots$ we have
\begin{align*}
\La f\ci{P}\Ra\ci{Q_k}&\leq
\strokedint_{P\times Q_k}|\cV\ci{x_2,P}(B(x_1,y_2)-\La B\Ra\ci{P\times Q_k})\cU\ci{x_2,P}^{-1}|\mathd x_1\mathd y_2\\
&+
\strokedint_{P\times Q_k}|\cV\ci{x_2,P}(\La B_{y_2}\Ra\ci{P}-\La B\Ra\ci{P\times Q_k})\cU\ci{x_2,P}^{-1}|\mathd x_1\mathd y_2.
\end{align*}
For all $k=1,2,\ldots$, we have
\begin{align*}
&\strokedint_{P\times Q_k}|\cV\ci{x_2,P}(B(x_1,y_2)-\La B\Ra\ci{P\times Q_k})\cU\ci{x_2,P}^{-1}|\mathd x_1\mathd y_2\\
&\sim_{p,d}\strokedint_{P\times Q_k}|\cV\ci{P,2}(B(x_1,y_2)-\La B\Ra\ci{P\times Q_k})\cU\ci{P,2}^{-1}|\mathd x_1\mathd y_2\\
&\leq\strokedint_{P\times Q_k}|\cV\ci{P\times Q_k}(B(x_1,y_2)-\La B\Ra\ci{P\times Q_k})\cU\ci{P\times Q_k}^{-1}|\mathd x_1\mathd y_2\\
&+|\cV\ci{P\times Q_k}|\cdot|\cU^{-1}\ci{P\times Q_k}-\cU^{-1}\ci{P,2}|\strokedint_{P\times Q_k}|B(x_1,y_2)-\La B\Ra\ci{P\times Q_k}|\mathd x_1\mathd y_2\\
&+|\cV\ci{P\times Q_k}-\cV\ci{P,2}|\cdot|\cU^{-1}\ci{P,2}|\strokedint_{P\times Q_k}|B(x_1,y_2)-\La B\Ra\ci{P\times Q_k}|\mathd x_1\mathd y_2\\
&\leq\Vert B\Vert\ci{\wt{\text{bmo}}\ci{\bfD}(U,V,p)}+2\La F\ci{P}\Ra\ci{Q_k}\big(|\cV\ci{P\times Q_k}|\cdot|\cU^{-1}\ci{P,2}-\cU^{-1}\ci{P\times Q_k}|+|\cV\ci{P\times Q_
k}-\cV\ci{P,2}|\cdot|\cU^{-1}\ci{P,2}|).
\end{align*}
We have $\lim_{k\rightarrow\infty}\cU\ci{P\times Q_k}=\cU\ci{P,2}$, so $\lim_{k\rightarrow\infty}\cU\ci{P\times Q_k}^{-1}=\cU\ci{P,2}^{-1}$. Moreover, we have
\begin{equation*}
\lim_{k\rightarrow\infty}\cV\ci{P\times Q_k}=\cV\ci{P,2},\qquad
\lim_{k\rightarrow\infty}\La F\ci{P}\Ra\ci{Q_k}=F\ci{P}(x_2)<\infty.
\end{equation*}
Therefore
\begin{equation*}
\limsup_{k\rightarrow\infty}\strokedint_{P\times Q_k}|\cV\ci{x_2,P}(B(x_1,y_2)-\La B\Ra\ci{P\times Q_k})\cU\ci{x_2,P}^{-1}|\mathd x_1\mathd y_2\lesssim_{p,d}\Vert B\Vert\ci{\wt{\text{bmo}}\ci{\bfD}(U,V,p)}.
\end{equation*}
Moreover, we have
\begin{align*}
&\strokedint_{P\times Q_k}|\cV\ci{x_2,P}(\La B_{y_2}\Ra\ci{P}-\La B\Ra\ci{P\times Q_k})\cU\ci{x_2,P}^{-1}|\mathd x_1\mathd y_2\\
&=\strokedint_{Q_k}|\cV\ci{x_2,P}(\La B_{y_2}\Ra\ci{P}-\La B\Ra\ci{P\times Q_k})\cU\ci{x_2,P}^{-1}|\mathd y_2\\
&\leq\strokedint_{P\times Q_k}|\cV\ci{x_2,P}(B(x_1,y_2)-\La B\Ra\ci{P\times Q_k})\cU^{-1}\ci{x_2,P}|\mathd x_1\mathd y_2,
\end{align*}
for all $k=1,2,\ldots$. We saw above that 
\begin{equation*}
\limsup_{k\rightarrow\infty}\strokedint_{P\times Q_k}|\cV\ci{x_2,P}(B(x_1,y_2)-\La B\Ra\ci{P\times Q_k})\cU\ci{x_2,P}^{-1}|\mathd x_1\mathd y_2\lesssim_{p,d}\Vert B\Vert\ci{\wt{\text{bmo}}\ci{\bfD}(U,V,p)},
\end{equation*}
concluding the proof.
\end{proof}

We can now prove Proposition \ref{prop:equivalencesdyadiclittleBMO}.

\begin{proof}[Proof (of Proposition \ref{prop:equivalencesdyadiclittleBMO})]
Estimate \eqref{mainBMOdyadicequivalence} follows immediately by combining Lemmas \ref{l: step1}, \ref{l: step2}, \ref{l: step3} and \ref{l: two-weight-biparameter A_p implies uniform A_p in each coordinate} with the results in the one-parameter case \cite[Lemma 4.7]{isralowitz-pott-treil}. Estimate \eqref{secondaryBMOdyadicequivalence} can then be proved using estimate \eqref{mainBMOdyadicequivalence} in the exact same way as this is done in the one-parameter case in \cite[Lemma 4.7]{isralowitz-pott-treil}; we omit the details.
\end{proof}

\subsection{Continuous versions} Let $1<p<\infty$. Let $U,V$ be biparameter $d\times d$-matrix valued $A_p$ weights on $\R^n\times\R^m$. Let $B\in L^1\ti{loc}(\R^{n+m};M_{d}(\C))$. We define the two-matrix weighted little BMO norm $\Vert B\Vert\ci{\text{bmo}(U,V,p)}$ by
\begin{equation}
\label{matrix two-weight little BMO}
\Vert B\Vert\ci{\text{bmo}(U,V,p)}:=\sup_{R}\left(\strokedint_{R}|V(x)^{1/p}(B(x)-\La B\Ra\ci{R})\cU\ci{R}^{-1}|^{p}\mathd x\right)^{1/p},
\end{equation}
where the supremum is taken over all rectangles $R$ in $\R^n\times\R^m$ (with sides parallel to the coordinate axes), and all reducing operators are taken with respect to exponent $p$. We also define
\begin{equation}
\label{matrix ``one-weight'' little BMO}
\Vert B\Vert\ci{\wt{\text{bmo}}(U,V,p)}:=\sup_{R}\strokedint_{R}|\cV\ci{R}(B(x)-\La B\Ra\ci{R})\cU\ci{R}^{-1}|\mathd x,
\end{equation}
where the supremum is taken over all rectangles $R$ in $\R^n\times\R^m$ (with sides parallel to the coordinate axes), and all reducing operators are again taken with respect to exponent $p$. As before we also need the variants
\begin{equation}
\label{matrix two-weight little BMO first pointwise}
\Vert B\Vert\ci{\text{bmo}^1(U,V,p)}:=\sup_{R}\left(\strokedint_{R}\left(\strokedint_{R}|V(x)^{1/p}(B(x)-B(y))U(y)^{-1/p}|^{p'}\mathd y\right)^{p/p'}\mathd x\right)^{1/p}
\end{equation}
and
\begin{equation}
\label{matrix two-weight little BMO second pointwise}
\Vert B\Vert\ci{\text{bmo}^2(U,V,p)}:=\sup_{R}\left(\strokedint_{R}\left(\strokedint_{R}|V(x)^{1/p}(B(x)-B(y))U(y)^{-1/p}|^{p}\mathd x\right)^{p'/p}\mathd y\right)^{1/p'},
\end{equation}
where both suprema are taken over all rectangles $R$ in $\R^n\times\R^m$.

By the well-known $1/3$ trick (see e.g. \cite{lerner-nazarov}) we have that there exist $3^n$ dyadic grids $\cD_1^{(n)},\ldots,\cD_{3^{n}}^{(n)}$ in $\R^{n}$, such that for all cubes $Q$ in $\R^n$, there exist $i\in\lbrace1,\ldots,3^{n}\rbrace$ and $Q'\in\cD^{(n)}_{i}$ with $Q\subseteq Q'$ and $\ell(Q')\leq 6\ell(Q)$. Thus, we can consider the product dyadic grids
\begin{equation*}
\bfD^{i,j}:=\cD_{i}^{(n)}\times\cD_{j}^{(m)},\qquad i=1,\ldots,3^n,~j=1,\ldots,3^m
\end{equation*}
in $\R^n\times\R^m$. To be able to use covering arguments in the matrix-weighted setting, we need the following trivial observation, which follows immediately from the definition of reducing operators.

\begin{lm}
\label{l: trivial inclusion reducing operators}
Let $E,F$ be measurable subsets of $\R^{n}$ with $E\subseteq F$ and $0<|E|,|F|<\infty$. Let $W\in L^1(F;M_{d}(\C))$. Then, there holds
\begin{equation*}
|\cW\ci{E}e|\lesssim_{p,d}\left(\frac{|F|}{|E|}\right)^{1/p}|\cW\ci{F}e|,\qquad\forall e\in\C^d.
\end{equation*}
\end{lm}

We can now extend Proposition \ref{prop:equivalencesdyadiclittleBMO} to the continuous setting.

\begin{cor}
\label{c: equivalences of little bmo}
There holds
\begin{equation*}
\Vert B\Vert\ci{\emph{bmo}(U,V,p)}\sim\max(\esssup_{x_2\in\R^m}\Vert B_{x_2}\Vert\ci{\emph{BMO}(U_{x_2},V_{x_2},p)},\esssup_{x_1\in\R^n}\Vert B_{x_1}\Vert\ci{\emph{BMO}(U_{x_1},V_{x_1},p)}),
\end{equation*}
\begin{equation*}
\Vert B\Vert\ci{\emph{bmo}(U,V,p)}\sim\Vert B^{\ast}\Vert\ci{\emph{bmo}(V',U',p')}\sim\Vert B\Vert\ci{\wt{\emph{bmo}}(U,V,p)}\sim\Vert B^{\ast}\Vert\ci{\wt{\emph{bmo}}(V',U',p')},
\end{equation*}
and
\begin{equation*}
\Vert B\Vert\ci{\emph{bmo}(U,V,p)}\sim\Vert B\Vert\ci{\emph{bmo}^{1}(U,V,p)}\sim\Vert B\Vert\ci{\emph{bmo}^{2}(U,V,p)},
\end{equation*}
where all implied constants depend only on $n,m,p,d,[U]\ci{A_p(\R^n\times\R^m)}$ and $[V]\ci{A_p(\R^n\times\R^m)}$.
\end{cor}

\begin{proof}
Let $R$ be any rectangle in $\R^n\times\R^m$. There exist $i\in\lbrace1,\ldots,3^n\rbrace$, $j\in\lbrace1,\ldots,3^m\rbrace$ and $S\in\bfD^{i,j}$, such that $R\subseteq S$ and $|S|\leq 6^{n+m}|R|$. Then, we have
\begin{align*}
&\left(\strokedint_{R}|V(x)^{1/p}(B(x)-\La B\Ra\ci{R})\cU\ci{R}^{-1}|^{p}\mathd x\right)^{1/p}\\
&\lesssim_{n,m,p,d}\left(\strokedint_{S}|V(x)^{1/p}(B(x)-\La B\Ra\ci{R})\cU'\ci{R}|^{p}\mathd x\right)^{1/p}\\
&\lesssim_{n,m,p,d}\left(\strokedint_{S}|V(x)^{1/p}(B(x)-\La B\Ra\ci{R})\cU'\ci{S}|^{p}\mathd x\right)^{1/p}\\
&\lesssim_{p,d}[U]\ci{A_p,\bfD^{i,j}}^{1/p}\left(\strokedint_{S}|V(x)^{1/p}(B(x)-\La B\Ra\ci{R})\cU^{-1}\ci{S}|^{p}\mathd x\right)^{1/p}\\
&\leq[U]\ci{A_p(\R^n\times\R^m)}^{1/p}\left(\strokedint_{S}|V(x)^{1/p}(B(x)-\La B\Ra\ci{S})\cU\ci{S}^{-1}|^{p}\mathd x\right)^{1/p}\\
&+[U]\ci{A_p(\R^n\times\R^m)}^{1/p}\left(\strokedint_{S}|V(x)^{1/p}(\La B\Ra\ci{S}-\La B\Ra\ci{R})\cU\ci{S}^{-1}|^{p}\mathd x\right)^{1/p},
\end{align*}
where in the first $\lesssim_{n,m,p,d}$ and in the first $\lesssim_{p,d}$ we used Lemma \ref{l: replace inverse by prime}, and in the second $\lesssim_{n,m,p,d}$ we used Lemma \ref{l: trivial inclusion reducing operators}. By definition we have
\begin{equation*}
\left(\strokedint_{S}|V(x)^{1/p}(B(x)-\La B\Ra\ci{S})\cU\ci{S}^{-1}|^{p}\mathd x\right)^{1/p}\leq\Vert B\Vert\ci{\text{bmo}\ci{\bfD^{i,j}}(U,V,p)}.
\end{equation*}
Moreover, we have
\begin{align*}
&\left(\strokedint_{S}|V(x)^{1/p}(\La B\Ra\ci{S}-\La B\Ra\ci{R})\cU\ci{S}^{-1}|^{p}\mathd x\right)^{1/p}\sim_{p,d}|\cV\ci{S}(\La B\Ra\ci{S}-\La B\Ra\ci{R})\cU\ci{S}^{-1}|\\
&\leq\strokedint_{R}|\cV\ci{S}(B(x)-\La B\Ra\ci{S})\cU\ci{S}^{-1}|\mathd x
\lesssim_{n,m}\strokedint_{S}|\cV\ci{S}(B(x)-\La B\Ra\ci{S})\cU\ci{S}^{-1}|\mathd x\\
&\leq\Vert B\Vert\ci{\wt{\text{bmo}}\ci{\bfD^{i,j}}(U,V,p)}
\sim\Vert B\Vert\ci{\text{bmo}\ci{\bfD^{i,j}}(U,V,p)},
\end{align*}
where in the last $\sim$ all implied constants depend only on $n,m,p, d, [U]\ci{A_p(\R^n\times\R^m)}$ and $[V]\ci{A_p(\R^n\times\R^m)}$.

Therefore
\begin{equation}
\label{little BMO from dyadic to continuous}
\Vert B\Vert\ci{\text{bmo}(U,V,p)}\sim\max\left\lbrace \Vert B\Vert\ci{\text{bmo}\ci{\bfD^{i,j}}(U,V,p)}:~i=1,\ldots,3^n,~j=1,\ldots,3^m\right\rbrace,
\end{equation}
where all implied constants depend only on $n,m,p, d, [U]\ci{A_p(\R^n\times\R^m)}$ and $[V]\ci{A_p(\R^n\times\R^m)}$. Similar arguments yield analogs of \eqref{little BMO from dyadic to continuous} for the other variants of matrix-weighted little BMO (as well as for one-parameter matrix-weighted BMO spaces), thus the desired estimates follow from Proposition \ref{prop:equivalencesdyadiclittleBMO}.
\end{proof}

\section{Matrix-weighted little BMO lower bounds}
\label{s: lower matrix little BMO bounds}
Here we show how to adapt the argument from \cite[Section~3]{isralowitz-pott-treil} to get a lower bound, in terms of the little BMO matrix norm, for commutators with all possible Riesz transforms.
To be more precise, let $R^1_j$ with $1 \leq j \leq n$ and $R^2_k$ with $1 \leq k \leq m$ denote, respectively, the Riesz transforms acting on the first and second variables for scalar functions defined on $\R^n\times\R^m.$
As before, we use the same notation for the canonical extensions of $R^1_j$ and $R^2_k$ acting componentwise on $\C^d$-valued functions defined on $\R^n\times\R^m$ (without risk of ambiguity as we will only use this canonical extension in the rest of this section).
Recall that, taking $x \in \R^n$ and $y \in \R^m,$ these operators have kernels
\begin{equation*}
  \frac{x_j-x_j'}{|x-x'|^{n+1}} \otimes I_d \quad \text{and} \quad \frac{y_k-y_k'}{|y-y'|^{m+1}} \otimes I_d.
\end{equation*}
Our goal is to show the following
\begin{thm}
  \label{thm:LowerBoundRiesz}
  Consider a function $B \in L^1\ti{loc}(\R^{n+m};M_{d}(\C)),$ let $1 < p < \infty$ and let $U$ and $V$ be $d \times d$-matrix valued biparameter $A_p$ weights.
  Then, there holds
  \begin{equation}
  \label{eq:RieszLowerBound}
  \Vert B\Vert\ci{\emph{bmo}(U,V,p)} \lesssim
  \max_{\substack{1\leq j\leq n\\ 1\leq k \leq m}} \Vert[R^1_jR^2_k,B]\Vert\ci{L^{p}(U)\rightarrow L^{p}(V)},
  \end{equation}
  where the implied constant depends only on the $n,m,p,d,[U]\ci{A_p,\bfD}$ and $[V]\ci{A_p,\bfD}$.
\end{thm}

It can look like the proof of this result takes a different approach to the classical one for such lower bounds, due to Coifman--Rochberg--Weiss \cite{coifman-rochberg-weiss} and based on a property of spherical harmonics.
However, the main idea of this approach is still at the core of the argument since we will also express the unit as a linear combination of products of Riesz transforms, together with the tensorization trick due to Isralowitz--Pott--Treil \cite{isralowitz-pott-treil}.
For a suitable scalar function $h$ defined on $\R^{n+m},$ let us denote by $\fourier{h}(t) \coloneq \int_{\R^{n+m}} h(x) e^{-i\langle t, x \rangle} \mathd x$ its Fourier transform.
Consider now the Wiener algebra $W_0(\R^{n+m}) \coloneq \{\psi = \fourier{\rho}\colon \rho \in L^1(\R^{n+m})\},$ that is the set of functions that are Fourier transforms of integrable functions.
First we need the following lemma, which follows from an observation in \cite[Lemma~2.1]{liaw-treil}.
For $x = (x_1,x_2), y = (y_1,y_2) \in \R^n\times\R^m,$ let us denote by $K_{jk}(x,y)$ the kernel of the product $R^1_jR^2.$
\begin{lm}
  \label{lm:WienerAlgebraLemma}
  Assume that $W$ is a $d \times d$-matrix valued weight and that $R^1_jR^2_k$ are bounded in $L^p(W)$ for $1 \leq j \leq n$ and $1 \leq k \leq m.$
  That is, for compactly supported $f \in L^2 \cap L^p(W)$ and $g \in L^2 \cap L^{p'}(W^{-p'/p})$ such that $\dist(\supp f, \supp g) > 0,$ if $E$ is a measurable set in $\R^{n+m},$ then
  \begin{multline}
  \label{eq:RieszLpBound}
  \left|\int_{\R^{n+m}}\int_{\R^{n+m}} \1\ci{E}(x,y)K_{jk}(x,y) \langle f(x), g(y) \rangle \mathd x\mathd y \right|\\
  \leq \Vert\1\ci{E}R^1_jR^2_k\1\ci{E}\Vert\ci{L^{p}(W)\rightarrow L^{p}(W)} \Vert f \Vert\ci{L^{p}(W)} \Vert g \Vert\ci{L^{p'}(W^{-p'/p})}.
  \end{multline}
  Then, if $\psi \in W_0(\R^{n+m}),$ $\varepsilon_1, \varepsilon_2 > 0$ and $E$ is a measurable set of $\R^{n+m},$ it holds that
  \begin{multline}
  \label{eq:WienerBilinearEstimate}
  \left|\int_{\R^{n+m}}\int_{\R^{n+m}} \psi\left( \frac{x_1-y_1}{\varepsilon_1}, \frac{x_2-y_2}{\varepsilon_2}\right)\1\ci{E}(x,y)K_{jk}(x,y) \langle f(x), g(y) \rangle \mathd x\mathd y \right|\\
  \lesssim \Vert\1\ci{E}R^1_jR^2_k\1\ci{E}\Vert\ci{L^{p}(W)\rightarrow L^{p}(W)} \Vert f \Vert\ci{L^{p}(W)} \Vert g \Vert\ci{L^{p'}(W^{-p'/p})}.
  \end{multline}
\end{lm}

\begin{proof}
  Since $K_{jk}(x,y)$ satisfies \eqref{eq:RieszLpBound}, so does the kernel $K_{jk}(x,y) e^{-i \langle a, x-y \rangle}$ for any $a \in \R^{n+m}.$
  In particular, taking $\rho \in L^1(\R^{n+m}),$ the kernel
  \begin{equation*}
    \fourier{\rho}(x-y) K_{jk}(x,y) = \int_{\R^{n+m}} \rho(a) K_{jk}(x,y) e^{-i \langle a, x-y \rangle} \mathd a
  \end{equation*}
  also satisfies \eqref{eq:RieszLpBound} (with $\Vert \rho \Vert\ci{L^1(\R^{n+m})}$ as a multiplicative constant).
  Now, consider $(x_1,x_2),(t_1,t_2) \in \R^{n+m}$ and observe that rescaling each variable separately we get
  \begin{equation*}
    \fourier{\rho}(t_1/\varepsilon_1, t_2/\varepsilon_2) = \varepsilon_1^n\varepsilon_2^m\fourier{\rho(\varepsilon_1x_1, \varepsilon_2x_2)}
  \end{equation*}
  and
  \begin{equation*}
    \int_{\R^n}\int_{\R^m} |\varepsilon_1^n\varepsilon_2^m \rho(\varepsilon_1x_1, \varepsilon_2x_2)| \mathd x_1 \mathd x_2 = \Vert \rho \Vert\ci{L^1(\R^{n+m})}.
  \end{equation*}
  Thus, letting $\rho$ to be such that $\psi = \fourier{\rho}$ and taking these observations into account, the kernel $\psi((x_1-y_1)/\varepsilon_1, (x_2-y_2)/\varepsilon_2) K_{jk}(x,y)$ satisfies equation \eqref{eq:RieszLpBound} (with the multiplicative constant $\Vert \rho \Vert\ci{L^1(\R^{n+m})}$), which is precisely \eqref{eq:WienerBilinearEstimate}.
\end{proof}

\begin{rem}
  Observe that if $W$ is a $d \times d$-matrix valued biparameter $A_2$ weight defined on $\R^{n+m},$ then the operators $R^1_jR^2_k$ (which are paraproduct-free) are bounded in $L^p(W)$ for $1 \leq j \leq n$ and $1 \leq k \leq m$ (see \cite{matrix journe}, and also \cite{Convex} in the case $p = 2$ and \cite{cruz-uribe-isralowitz-moen} in the case of general $p$ for the one-parameter results).
\end{rem}

The next lemma is just a reformulation of \cite[Lemma~3.2]{isralowitz-pott-treil} for it to be directly applicable to the biparameter setting.
We include its proof, which is elementary, for the reader's convenience.
\begin{lm}
  \label{lm:ModulusMultiplicationWiener}
  If $\varphi \in C^{\infty}\ti{c}(\R^{n+m}),$ then the three functions $|x|\varphi(x,y),$ $|y|\varphi(x,y)$ and $|x||y|\varphi(x,y)$ are in $W_0(\R^{n+m}).$
\end{lm}
\begin{proof}
  The same as in \cite[Lemma~3.2]{isralowitz-pott-treil}.
  We show first the case of $|x||y|\varphi(x,y),$ as the other two are treated similarly, and then we indicate the only difference for those.
  Let $F(x,y) = |x||y|\varphi(x,y)$ and take $1 < \delta < \min \{1+1/(n+m-1),2\}.$
  For $\alpha \in \{0,1\}^n$ and $\beta \in \{0,1\}^m$ define
  \begin{equation*}
    P_{\alpha\beta} = \{(x,y) \in \R^n\times\R^m\colon (-1)^{\alpha_j} |x_{\alpha_j}| \leq (-1)^{\alpha_j}, (-1)^{\beta_k} |y_{\beta_k}| \leq (-1)^{\beta_k}\}.
  \end{equation*}
  Then
  \begin{align*}
    \Vert \fourier{F} \Vert\ci{L^1(\R^{n+m})}
    &= \sum_{\alpha,\beta} \int_{P_{\alpha\beta}} |x^\alpha|^{-1} |y^\beta|^{-1} \left(|x^\alpha| |y^\beta| |\fourier{F}(x,y)|\right) \mathd x \mathd y\\
    &\leq \sum_{\alpha,\beta} \left(\int_{P_{\alpha\beta}} |x^\alpha|^{-\delta} |y^\beta|^{-\delta} \mathd x \mathd y\right)^{1/\delta}
                              \left(\int_{\R^{n+m}} |x^\alpha y^\beta \fourier{F}(x,y)|^{\delta'} \mathd x \mathd y\right)^{1/\delta'}\\
    &\lesssim \sum_{\alpha,\beta} \left(\int_{\R^{n+m}} |\fourier{D_1^\alpha D_2^\beta F}(x,y)|^{\delta'} \mathd x \mathd y\right)^{1/\delta'}\\
    &\lesssim \sum_{\alpha,\beta} \left(\int_{\R^{n+m}} |D_1^\alpha D_2^\beta F(x,y)|^{\delta}\right)^{1/\delta},
  \end{align*}
  where we used the Hausdorff--Young Inequality in the last step and the symbols $D_1^\alpha$ and $D_2^\beta$ to denote, respectively, the $\alpha$-partial on the first variable and the $\beta$-partial derivative on the second variable.
  The result now follows from the estimate
  \begin{equation*}
    |D_1^\alpha D_2^\beta F(x,y)| \lesssim |(x,y)|^{1-|\alpha|-|\beta|},
  \end{equation*}
  the fact that $\varphi$ is compactly supported and that $1 < \delta < 1+1/(n+m-1)$ and Fourier inversion.
  
  For the case of $|x|\varphi(x,y)$ and $|y|\varphi(x,y),$ the computation is exactly the same.
  The only difference is that at the end one has the estimates $|D_1^\alpha D_2^\beta F(x,y)| \lesssim |x|^{1-|\alpha|}$ and $|D_1^\alpha D_2^\beta F(x,y)| \lesssim |y|^{1-|\beta|}$ respectively.
  Thus, in these two cases it would be enough to take $\delta < 1 + 1/(n-1)$ and $\delta < 1 + 1/(m-1)$ accordingly.
\end{proof}

The last ingredient to prove Theorem \ref{thm:LowerBoundRiesz} is the following lemma, which gives a local lower bound in terms of reducing operators.
\begin{lm}
\label{lm:ReducingOperatorLowerBound}
  Let $W$ be a $d \times d$-matrix valued weight and assume that $R^1_jR^2_k$ are bounded in $L^p(W)$ for $1 \leq j \leq n$ and $1 \leq k \leq m.$
  If $R$ is a rectangle in $\R^{n+m}$ with sides parallel to the coordinate axes, then
  \begin{equation*}
    |\cW_R'\cW_R| \lesssim \max_{\substack{1\leq j\leq n\\ 1\leq k \leq m}} \Vert\1\ci{R}R^1_jR^2_k\1\ci{R}\Vert\ci{L^{p}(W)\rightarrow L^{p}(W)}.
  \end{equation*}
\end{lm}
\begin{proof}
  Write $R = I \times J$ and let $\varepsilon_1, \varepsilon_2$ be the side-lengths of $I$ and $J$ respectively.
  Let $\varphi \in C^\infty\ti{c}(\R^{n+m})$ be such that $\varphi(x,y) = 1$ if $(x,y) \in [-2,2]^{n+m}.$
  Fix $1 \leq j \leq n$ and $1 \leq k \leq m.$
  Note that if both $n$ and $m$ are odd, then $x_j y_k |x|^{n-1} |y|^{m-1} \varphi^2(x,y) \in C^\infty\ti{c}(\R^{n+m}) \subseteq W_0(\R^{n+m}).$
  On the other hand, if both $n$ and $m$ are even, then $x_j y_k |x|^{n-2} |y|^{m-2} \varphi(x,y) \in C^\infty\ti{c}(\R^{n+m}) \subseteq W_0(\R^{n+m}),$ so that Lemma \ref{lm:WienerAlgebraLemma} implies that $x_j y_k |x|^{n-1} |y|^{m-1} \varphi^2(x,y) \in W_0(\R^{n+m}).$
  The case in which $n$ is even and $m$ is odd (and vice versa) is also treated with Lemma \ref{lm:WienerAlgebraLemma} to get that $x_j y_k |x|^{n-1} |y|^{m-1} \varphi^2(x,y) \in W_0(\R^{n+m}).$
  Next, observe that we can rewrite the averaging operator kernel as
  \begin{align*}
    \frac{\1_{R\times R}((x,y),(x',y'))}{|R|}
    &\sim \varepsilon_1^{-n} \varepsilon_2^{-m} \varphi\left(\frac{x-x'}{\varepsilon_1},\frac{y-y'}{\varepsilon_2}\right) \1_{R\times R}((x,y),(x',y'))\\
    &= \left(\sum_{j,k} \frac{x_j-x_j'}{\varepsilon_1} \frac{y_k-y_k'}{\varepsilon_2} \left|\frac{x-x'}{\varepsilon_1}\right|^{n-1} \left|\frac{y-y'}{\varepsilon_2}\right|^{m-1} K_{jk}((x,y),(x',y'))\right)\\
    &\times \varphi^2\left(\frac{x-x'}{\varepsilon_1},\frac{y-y'}{\varepsilon_2}\right) \1_{R\times R}((x,y),(x',y')),
  \end{align*}
  so that each of the terms in this sum satisfies \eqref{eq:WienerBilinearEstimate}.
  Hence, for any compactly supported $f \in L^2 \cap L^p(W)$ and $g \in L^2 \cap L^{p'}(W^{-p'/p}),$ the averaging operator satisfies
  \begin{equation*}
    \left|(A_R f, g )\right| \lesssim \max_{\substack{1\leq j\leq n\\ 1\leq k \leq m}} \Vert\1\ci{R}R^1_jR^2_k\1\ci{R}\Vert\ci{L^{p}(W)\rightarrow L^{p}(W)} \Vert f \Vert\ci{L^{p}(W)} \Vert g \Vert\ci{L^{p'}(W^{-p'/p})}.
  \end{equation*}
  Using a standard density argument, it follows that
  \begin{equation*}
    |\cW_R'\cW_R| \sim \Vert A_R \Vert\ci{L^{p}(W)\rightarrow L^{p}(W)} \lesssim \max_{\substack{1\leq j\leq n\\ 1\leq k \leq m}} \Vert\1\ci{R}R^1_jR^2_k\1\ci{R}\Vert\ci{L^{p}(W)\rightarrow L^{p}(W)},
  \end{equation*}
  where in the $\sim$ we applied \cite[Lemma 3.2]{matrix journe}.
\end{proof}

\begin{proof}[Proof of Theorem \ref{thm:LowerBoundRiesz}]
  Assume that the right-hand side of \eqref{eq:RieszLowerBound} is finite, since otherwise it is trivial.
  For a matrix valued weight $W$ and a rectangle $R\subseteq\R^n\times\R^m,$ let us denote
  \begin{equation*}
    [W]\ci{A_p,R} \coloneq \strokedint_{R}\left(\strokedint_{R}|W(x)^{1/p}W(y)^{-1/p}|^{p'}\mathd y \right)^{p/p'}\mathd x,
  \end{equation*}
  and for a matrix valued function $B \in L^1\ti{loc}(\R^{n+m};M_{d}(\C))$ let us denote
  \begin{equation*}
    \Vert B\Vert\ci{\wt{\text{bmo}}^1(U,V,p),R} \coloneq \left(\strokedint_{R}\left(\strokedint_{R}|V(x)^{1/p}(B(x)-B(y))U(y)^{-1/p}|^{p'}\mathd y\right)^{p/p'}\mathd x\right)^{1/p}.
  \end{equation*}
  Define the $2d \times 2d$-matrix valued function
  \begin{equation*}
    \Phi = \begin{pmatrix}
             V^{1/p} & V^{1/p}B \\
             0       & U^{1/p}
           \end{pmatrix},
  \end{equation*}
  that has a.e. inverse
  \begin{equation*}
    \Phi^{-1} = \begin{pmatrix}
                  V^{-1/p} & -BU^{-1/p} \\
                  0        & U^{-1/p}
                  \end{pmatrix}
  \end{equation*}
  and let $W^{1/p} = (\Phi^\ast \Phi)^{1/2}.$
  By the polar decomposition we have that $\Phi = \cU W^{1/p},$ where $\cU$ is an a.e. unitary matrix.
  Note that a direct computation shows that
  \begin{align*}
    \strokedint_R \bigg(\strokedint_R &\Vert W^{1/p}(x)W^{-1/p}(y) \Vert^{p'}\mathd y\bigg)^{p/p'} \mathd x\\
    &\sim [V]_{A_p,R} + [U]_{A_p,R} + \Vert B\Vert\ci{\wt{\text{bmo}}^1(U,V,p),R}^p.
  \end{align*}
  Now, use that for a linear operator $T$ acting on $\C^d$-valued functions it happens that
   \begin{equation*}
    \Phi (T \otimes I_2) \Phi^{-1} = \begin{pmatrix}
                                       V^{1/p}(T \otimes I_2)V^{-1/p} & V^{1/p}[B,T]U^{-1/p} \\
                                       0                              & U^{1/p}(T \otimes I_2)U^{-1/p}
                                     \end{pmatrix},
  \end{equation*}
  so that by Lemma \ref{lm:ReducingOperatorLowerBound} we obtain
  \begin{align*}
    \big([V]_{A_p,R} + [U]_{A_p,R} &+ \Vert B\Vert\ci{\wt{\text{bmo}}^1(U,V,p),R}^p\big)^{1/p}\\
    &\sim [W]_{A_p,R}^{1/p} \sim |\cW_R'\cW_R|\\
    &\lesssim \max_{\substack{1\leq j\leq n\\ 1\leq k \leq m}} \Vert\1\ci{R}R^1_jR^2_k\1\ci{R}\Vert\ci{L^{p}(W)\rightarrow L^{p}(W)}\\
    &\leq \max_{\substack{1\leq j\leq n\\ 1\leq k \leq m}} \Vert W^{1/p}R^1_jR^2_kW^{-1/p} \Vert\ci{L^{p}\rightarrow L^{p}}\\
    &= \max_{\substack{1\leq j\leq n\\ 1\leq k \leq m}} \Vert \Phi R^1_jR^2_k\Phi^{-1} \Vert\ci{L^{p}\rightarrow L^{p}}\\
    &\lesssim \max_{\substack{1\leq j\leq n\\ 1\leq k \leq m}} \big(\Vert[B,R^1_jR^2_k]\Vert\ci{L^{p}(U)\rightarrow L^{p}(V)}\\
              &+ \Vert R^1_jR^2_k \Vert\ci{L^{p}(U)\rightarrow L^{p}(U)} + \Vert R^1_jR^2_k \Vert\ci{L^{p}(V)\rightarrow L^{p}(V)}\big).
  \end{align*}
  Since both $U$ and $V$ are biparameter $A_p$ weights, all quantities above are finite.
  Thus, rescaling $B \mapsto rB,$ dividing by $r$ and letting $r \to \infty$ we get
  \begin{equation*}
    \Vert B\Vert\ci{\wt{\text{bmo}}^1(U,V,p),R} \lesssim \max_{\substack{1\leq j\leq n\\ 1\leq k \leq m}} \Vert[B,R^1_jR^2_k]\Vert\ci{L^{p}(U)\rightarrow L^{p}(V)}.
  \end{equation*}
  Finally, taking the supremum over all rectangles and using Corollary \ref{c: equivalences of little bmo}, the result follows.
\end{proof}

\section{Matrix-weighted little BMO upper bounds}
\label{s: upper matrix little BMO bounds}

In this section, we briefly explain how one can get upper bounds in terms of matrix-weighted little BMO norms for commutators with (a wide classes of) Journ\'e operators. The main tool we need is the following lemma. It is the natural biparameter counterpart of \cite[Lemma 1.3]{isralowitz-pott-treil} due to Isralowitz--Pott--Treil. Since the proof in the present two-parameter setting is exactly the same as in the one-parameter setting of \cite{isralowitz-pott-treil} (up to replacing the word ``cube'' with the word ``rectangle''), we omit it and instead refer the reader to \cite[Section 2]{isralowitz-pott-treil}.

\begin{lm}
\label{l: upper bounds}
Let $T$ be a linear operator acting on scalar-valued functions on $\R^n\times\R^m$. Let $1<p<\infty$. Assume that there exists an increasing function $\phi:[0,\infty)\rightarrow[0,\infty)$ (possibly depending on $T,n,m,p,d$) such that for all biparameter $(2d)\times(2d)$-matrix valued $A_p$ weights $W$ on $\R^n\times\R^m$, the canonical vector-valued extension of $T$ (denoted again by $T$) satisfies
\begin{equation*}
\Vert T\Vert\ci{L^{p}(W)\rightarrow L^{p}(W)}\leq\phi([W]\ci{A_p(\R^n\times\R^m)}).
\end{equation*}
Then, for all biparameter $d\times d$-matrix valued $A_p$ weights $U,V$ on $\R^n\times\R^m$ and for all $B\in L^1\ti{loc}(\R^{n+m};M_{d}(\C))$ there holds
\begin{equation*}
\Vert[T,B]\Vert\ci{L^{p}(U)\rightarrow L^{p}(V)}\leq\Vert B\Vert\ci{\emph{bmo}^1(U,V,p)}\phi(3^{p-1}([U]\ci{A_p(\R^n\times\R^m)}+[V]\ci{A_p(\R^n\times\R^m)}+1)).
\end{equation*}
\end{lm}

Using Lemma \ref{l: upper bounds} coupled with Corollary \ref{c: equivalences of little bmo} and already known weighted estimates from \cite{vuorinen}, we can deduce the main result.

\theoremstyle{plain}
\newtheorem*{thm_MatrixBMOUpperBound}{Theorem~\ref{thm:MatrixBMOUpperBound}}
\begin{thm_MatrixBMOUpperBound}
Let $T$ be any Journ\'e operator on $\R^n\times\R^m$. Let $1<p<\infty$, let $U,V$ be biparameter $d\times d$-matrix valued $A_p$ weights on $\R^n\times\R^m$, and let $B\in L^1\ti{loc}(\R^{n+m};M_d(\C))$. Then
\begin{equation*}
\Vert[T,B]\Vert\ci{L^{p}(U)\rightarrow L^{p}(V)}\lesssim\Vert B\Vert\ci{\emph{bmo}(U,V,p)},
\end{equation*}
where the implied constant depends only on $T,n,m,p,d,[U]\ci{A_p(\R^n\times\R^m)}$ and $[V]\ci{A_p(\R^n\times\R^m)}$.
\end{thm_MatrixBMOUpperBound}

\end{document}